\documentclass[10pt]{amsart}
\usepackage{amscd}
\usepackage{amssymb}
\usepackage{amsmath}
\usepackage{amsfonts}
\usepackage{enumerate}
\usepackage[all,cmtip]{xy}
\usepackage{tikz-cd}
\usepackage{hyperref}
\usepackage{xcolor}
\usepackage{adjustbox}
\usepackage{graphicx}

\newtheorem{thm}{Theorem}[section]
\newtheorem{prop}[thm]{Proposition}
\newtheorem{lem}[thm]{Lemma}
\newtheorem{cor}[thm]{Corollary}

\theoremstyle{definition}
\newtheorem{Def}[thm]{Definition}
\newtheorem{exmp}[thm]{Example}

\newtheorem{rem}[thm]{Remark}

\DeclareMathOperator*{\colim}{colim}
\newcommand{\Ss}{\mathcal {S}}
\newcommand{\G}{GL_1(R)}

\newcommand{\M}{\mathrm{Mod}_R}
\newcommand{\Si}{\Sigma^{\infty}_+}
\newcommand{\T}{\mathrm{Th}}
\newcommand{\E}{\mathbb{E}_{\infty}}
\newcommand{\Tot}{\mathrm{Tot}}
\newcommand{\threerightarrows}{%
\mathrel{\vcenter{\hbox{$\rightrightarrows$}\nointerlineskip\hbox{$\rightarrow$}}}
}
\newcommand{\fourrightarrows}{%
\mathrel{\vcenter{\hbox{$\threerightarrows$}\nointerlineskip\hbox{$\rightarrow$}}}
}

\keywords{topological coHochschild homology, Thom spectra.}

\numberwithin{equation}{section}

\title{Topological coHochschild homology and  Thom spectra}

\author[J. Zha]{Jiaxi Zha}
\address{Department of Mathematics, Nankai University, No.94 Weijin Road, Tianjin 300071, P. R. China}
\email{1093913699@qq.com}

\begin{document}

\begin{abstract}
   For an $\E$-ring spectrum $R$ and a map $f:X\to Pic(R)$ of spaces, the Thom spectrum $\T f$ is a comodule over $R\otimes\Si X$. In this paper we study the topological coHochschild homology of $R\otimes\Si X$ with coefficient $\T f$. More concretely, for a simply connected space $X$, we will give a filtration on $\mathrm{coTHH}^R(\T f;R\otimes\Si X)$ via the cellular structure of $X$. Furthermore, we will reduce the computation of $\mathrm{coTHH}^R(\T f;R\otimes\Si X)$ to that of $\mathrm{coTHH}^R(R\otimes\Si G)$ for some group-like $\mathbb{E}_1$-spaces. Finally, we will use these results to study properties of $\mathrm{coTHH}^R(\T f;R\otimes\Si X)$.
\end{abstract}

\maketitle

\section{Introduction}

Just as topological Hochschild homology (THH) generalizes Hochschild homology (HH) from  discrete algebras to ring spectra, topological coHochschild homology (coTHH) generalizes coHochschild homology (coHH) from discrete coalgebras to coalgebra spectra. As an invariant of discrete coalgebras, coHH was first introduced by Doi in \cite{Doi81}, and was later extended to differential graded coalgebras by Hess--Parent--Scott in \cite{HPS09}. It is noteworthy that in the same paper, the authors also established a quasi-isomorphism $\mathrm{coHH}(BA)\simeq \mathrm{HH}(A)$ for a non-negative augmented chain algebra $A$, where $BA$ is the bar construction for $A$. More recently, dual to THH for ring spectra, Hess--Shipley \cite{HS21} introduced coTHH for coalgebra spectra. Hess--Shipley \cite{HS21} also proved that there is an equivalence
    \[\mathrm{coTHH}(\Si X)\simeq \Si LX\]
    for simply connected spaces where $LX$ is the free loop space of $X$. Consequently, similar to the quasi-isomorphism in \cite{HPS09}, one obtains an instance of Koszul duality between THH and coTHH: $\mathrm{coTHH}(\Si X)\simeq \mathrm{THH}(\Si\Omega X)$. Beyond the Koszul duality, coTHH can also be related to THH via the Spanier--Whitehead duality under certain finiteness assumptions. In fact, it follows immediately from the definition that if $C$ is an $\mathbb{E}_1$-coalgebra whose underlying spectrum is dualizable, then $C^{\vee}$ is an $\mathbb{E}_1$-algebra and $\mathrm{coTHH}(C)\simeq \mathrm{THH}(C^{\vee})^{\vee}$. Further details on the connection between the Spanier--Whitehead duality and coTHH can be found in \cite{BP23}. There has already been a series of computations on coTHH. A key tool in the computation of coTHH, as that of THH, is the coB\"okstedt spectral sequence developed by Bohmann--Gerhardt--H\o genhaven--Shipley--Ziegenhagen in \cite{BGT+18}. The coB\"okstedt spectral sequence was further studied by Bohmann--Gerhardt--Shipley in \cite{BGS22}, where they proved that, under certain coflatness conditions, it admits a so-called $\square$-Hopf algebra structure. Using this structure, they generalized earlier homology computations of \cite{KY97} for various free loop spaces.

More generally, we can consider coTHH in general symmetric monoidal categories. Since the opposite category of a symmetric monoidal category is still symmetric monoidal, all constructions above are special cases of THH for algebras in a symmetric monoidal category developed by Nikolaus and Scholze in \cite{NS18}. Similarly, the structure on THH with coefficients is axiomatized in \cite{KMN23}. In \cite{Blu10-1} and \cite{Blu10-2}, Blumberg provided a detailed investigation of THH of Thom spectra. He proved that, for instance, for a $3$-fold loop map $f\colon X\to B\G$ with $X$ connected, there is an equivalence
\begin{equation}\label{eq:THH}
	\mathrm{THH}(\T f/R)\simeq \T f\otimes \Si BX.
\end{equation}
Although $\T f$ is always an $R$-algebra whenever $f$ is an $\mathbb{E}_1$-map, it is generally not a coalgebra. However, $\T f$ is always an $R\otimes\Si X$-comodule. In this paper, we consider coTHH with Thom spectra coefficients, which, as noted above, can be viewed as THH with coefficients in the opposite category of $\M$. Denote coTHH of $\Si X$ with the coefficient $\T f$ for a map $f\colon X\to Pic(R)$ by $\mathrm{coTHH}^R(\T f;R\otimes\Si X)$. We will show that, when $X$ is a simply connected $\mathbb{E}_1$-space, the dual to Equivalence \ref{eq:THH} holds (Proposition \ref{prop:omega}):
\[\mathrm{coTHH}^R(\T f;R\otimes\Si X)\simeq \T f\otimes \Si \Omega X.\]
Consequently, we restrict our attention to the case of general simply connected spaces. For this purpose, we need to understand the comodule structure of the Thom spectrum. This has been studied by Beardsley in \cite{Bea23}, in which he considered $\mathbb{E}_n$-ring spectra for $1\leq n\leq \infty$. We restrict our attention here to $\E$-ring spectra and connected spaces, although our results can be extended to the $\mathbb{E}_n$-setting and non-connected spaces. We will also prove that the comodule structure we obtain coincides with the one in \cite{Bea23}, see Proposition \ref{prop:compare}. 

\begin{thm}\label{thm:comodule}
	Let $R$ be an $\E$-ring spectrum and $X$ be a connected space. Given a map of spaces $f\colon X\to B\G$, the Thom spectrum $\T(R_X)\simeq R\otimes\Si X$ is an $\E $-R-coalgebra, and $\T f$ is a $\T(R_X)$-comodule, where $R_X$ is the constant local system taking the value $R$.
\end{thm}

Consequently, we can consider topological coHochschild homology of $R\otimes\Si X$ with coefficient $\T f$. An important property for $\T f$ is that if $f$ is $A$-oriented for some $R$-algebra $A$, then there is an equivalence of $A$-modules $A\otimes_R\T f\simeq A\otimes\Si X$. We prove that this property is compatible with coTHH with Thom spectra coefficients in the following sense.

\begin{thm}\label{thm:orient}
	Let $R$ be a connective $\E$-ring spectrum and $X$ be a simply connected space. Given a connective $\E$-$R$-algebra $A$ and a map $f\colon X\to B\G$, there is an equivalence
	\[A\otimes_R\mathrm{coTHH}^R(\T f;R\otimes\Si X)\simeq A\otimes\mathrm{coTHH}(\Si X)\]
	if $f$ is $A$-oriented. Here we write $\mathrm{coTHH}(\Si X)$ for $\mathrm{coTHH}^{\mathbb{S}}(\Si X)$.
\end{thm}

An important example of a Thom spectrum is the homotopy orbit: let $R$ be an $\E$-ring spectrum with a $G$-action, regarded as a functor $R:BG\to \M$. By definition, its $G$-orbit is the colimit of this functor, hence $R_G$ is a Thom spectrum. This is because $R$, as a module over itself, is invertible. Therefore, we can also consider $\mathrm{coTHH}^R(R_G;R\otimes\Si BG)$ as an invariant of $R_G$. The following theorem deals with the case $G=\Omega X$ for a simply connected space $X$. We construct a filtration on $\mathrm{coTHH}^R(\T f;R\otimes\Si X)$ using the cellular structure of $X$. The same method can be applied to more general cases to construct analogous filtrations on coTHH with Thom spectra coefficients.

\begin{thm}\label{thm:S1}
	Let $R$ be a connective $\E$-ring and $X$ be a simply connected space. Given a map $f\colon X\to B\G$, the $R$-module $\mathrm{coTHH}^R(\T f;R\otimes\Si X)$ admits an exhaustive filtration 
	\[\{\mathrm{coTHH}^R(\T (f|_{X_n});R\otimes\Si X)\}_{n\geq0}.\]
	Furthermore, the $n$-th associated graded of this filtration is given by
	\[gr_n\simeq \bigoplus_{X(n)}\Sigma^{n}R\otimes \Si \Omega X,\ n\geq0,\]
	where $X(n)$ is the number of $n$-cells in $X$.
\end{thm}

Finally, for a connected space $X$, there is an equivalence $X\simeq B\Omega X$. Therefore, the Thom spectrum of a map $f:X\to B\G$ with $X$ connected can also be viewed as an $\Omega X$-orbit of $R$. In Section $4$, we will prove that $\T f$ is equivalent to $R\otimes_{R[\Omega X]}R$ in this case. Using this equivalence, we can reduce the computation of coTHH with Thom spectra coefficients to coTHH of group rings which are regarded as $\E$-$R$-coalgebras.

\begin{thm}\label{thm:geo}
	Let $R$ be a connective $\E $-ring spectrum and $f\colon X\to B\G$ be a map of spaces with $X$ simply connected, then there is an equivalence
	\[\mathrm{coTHH}^R(\T f;R\otimes\Si X)\simeq \colim\limits_{\Delta^{op}}\mathrm{coTHH}^R(R\otimes \Si G^n)\]
    where $G$ is the loop space of $X$ and $\{\mathrm{coTHH}^R(R\otimes \Si G^n)\}_{n\geq 0}$ is a simplicial $R$-module such that in degree $n$: 
	\begin{enumerate}[(a)]
		\item the face map $d_0$ is induced by the action of $G$ on $R$,
		\item the face maps $d_i$ are induced by the multiplication of $G$ for $0< i < n$,
		\item the face map $d_n$ is induced by the trivial action of $G$ on $R$, or equivalently, by the augmentation map $R\otimes\Si G \to R$,
		\item the degeneracy maps are induced by the unit map $R\to R\otimes\Si  G$.
	\end{enumerate} 
\end{thm}

\subsection{Organization}
In Section \ref{sec2}, following \cite{ABG+}, we recall the definition of the Thom spectrum. Then we prove that the Thom spectrum $\T f$ associated to a map $f\colon X\to Pic(R)$ is an $R\otimes\Si X$-comodule and compare this comodule structure with that of \cite{Bea23}. In Section \ref{sec3}, we first compute $A$-homology of $\mathrm{coTHH}^R(\T f;R\otimes\Si X)$ for an $\E$-$R$-algebra $A$ and prove Theorem \ref{thm:orient}. We will also prove Theorem \ref{thm:S1} in this section. Finally, the proof of Theorem \ref{thm:geo} and some of its applications are given in Section \ref{sec4}.

\subsection{Conventions}
We work with $\infty$-categories throughout, following \cite{HTT,HA}. Therefore, all categories are understood to be $\infty$-categories. Moreover,
    \begin{enumerate}
        \item For an ordinary category $\mathcal{C}$, we regard it as an $\infty$-category through the nerve functor, and by abuse of notation we still denote it by $\mathcal{C}$.
    	\item For a category $\mathcal{C}$, the totalization of a cosimplicial object $X_*$ of $\mathcal{C}$ is denoted by $\mathrm{Tot}(X_*)$, while the limit restricted to \(\Delta_{\leq n}\) is denoted by $\text{Tot}^n(X_*)$. Similarly, the geometric realization of a simplicial object $Y_{\bullet}$ is denoted by $|Y_{\bullet}|$, while the colimit restricted to $\Delta^{op}_{\leq n}$ is denoted by $sk_nY_{\bullet}$. Note that cosimplicial objects are subscripted with $*$, while simplicial objects are subscripted with $\bullet$.  
    	\item The category $\Ss$ is	the category of spaces and is always equipped with the Cartesian symmetric monoidal structure.
    	\item $\E$-ring spectra are abbreviated as $\E$-rings. Given an $\E$-ring $R$ and a space $X$, the constant functor from $X$ to $\M$ with value $R$ is denoted by $R_X$. Let $A$ be an $\E$-$R$-algebra, the base change functor $-\otimes_R A$ is denoted by $\mathrm{Ind}_R^A$.
    	\item For notational uniformity, we write $R[X]$ for $R\otimes\Si X$ which is the colimit of the functor $R_X$, or equivalently, the tensor product of the spectrum $R$ and the space $X$.
    	\item Finally, a spectrum $X$ is called $n$-connective if $\pi_i(X)=0$ for $i<n$. It is said to be connective if it is $0$-connective.
    \end{enumerate}

\subsection*{Acknowledgements}
This research received no specific grant from any funding agency in the public, commercial, or not-for-profit sectors. The author would like to express sincere gratitude to Professor Xiangjun Wang for his invaluable guidance and insightful discussions throughout this work.  

\section{Thom spectra and comodules}\label{sec2}
In this section, we recall the definition of Thom spectra and show that any Thom spectrum is a comodule for its corresponding ``trivial" Thom spectrum. First, recall that $\Ss$ is the unit in $\mathrm{Pr}^L$ which is equipped with Lurie's tensor product (see \cite[Remark 4.8.1.20]{HA}). Hence for any presentable symmetric monoidal category $\mathcal{C}$, there is a unique colimit-preserving symmetric monoidal functor $\Ss \to \mathcal{C}$, given by $X \mapsto \colim_X 1_\mathcal{C}$, where $1_{\mathcal{C}}$ is the unit of $\mathcal{C}$. In particular, there is a functor $\Ss[-]\colon\mathrm{CAlg}(\Ss) \to \mathrm{CAlg}(\mathrm{Cat}_{\infty})$ which admits a right adjoint with $\mathrm{Cat}_{\infty}$ equipped with the Cartesian symmetric monoidal structure. Furthermore, $\mathrm{CAlg}(\Ss)$ contains a coreflective subcategory, $\mathrm{CAlg}^{\mathrm{gp}}(\Ss)$, consisting of group-like commutative monoids.
Now the functor $Pic\colon\mathrm{CAlg}(\mathrm{Cat}_{\infty}) \to \mathrm{CAlg}^{\mathrm{gp}}(\Ss)$ is the right adjoint to the composition
\[\Ss[-]\colon \mathrm{CAlg}^{\mathrm{gp}}(\Ss) \hookrightarrow \mathrm{CAlg}(\Ss) \to \mathrm{CAlg}(\mathrm{Cat}_{\infty}).\]
The counit induces a symmetric monoidal functor $\Ss[Pic({\mathcal{C}})] \to \mathcal{C}$. As an object of $\Ss$ all morphisms in $Pic(\mathcal{C})$ are invertible, hence there is an equivalence $\Ss[Pic({\mathcal{C}})]\simeq Pic(\mathcal{C})$ by \cite[Theorem 4.4]{Ber24}. In particular, fix an $\E $-ring $R$ and let $\mathcal{C}$ be $\M$, the category of $R$-modules. We obtain a symmetric monoidal functor $Pic(R) \to \M$, where $Pic(\M)$ is denoted by $Pic(R)$ for brevity.

\begin{rem}
	By \cite[Theorem 4.4]{Ber24}, the canonical colimit-preserving symmetric monoidal functor $\Ss \to \mathrm{Cat}_{\infty}$ is the inclusion of $\infty$-groupoids into categories, which admits a right adjoint given by taking maximal sub-groupoids. Consequently, $Pic(\mathcal{C})$ is equivalent to $(\mathcal{C}^{\simeq})^{\times}$, where $(-)^{\times}$ is the right adjoint of the inclusion $\mathrm{CAlg}^{\mathrm{gp}}(\Ss) \to \mathrm{CAlg}(\Ss)$, i.e., taking maximal group-like sub-groupoids. Moreover, the counit is the inclusion of $(\mathcal{C}^{\simeq})^{\times}$ into $\mathcal{C}$.
\end{rem}

\begin{Def}\cite[Definition 1.4]{ABG+}
	Let $f\colon X \to Pic(R) \in \Ss_{/Pic(R)}$ be a local system of invertible $R$-modules, the Thom spectrum of $f$ is defined to be the $R$-module spectrum
	\[\T f:= \mathrm{colimit}(X \to Pic(R) \to \M).\]
	The connected component of $Pic(R)$ containing the unit object $R$ is denoted by $B\G$.
\end{Def}

\begin{rem}\label{rm:col}
	Since colimits in $\Ss_{/Pic(R)}$ are computed in $\Ss$, the Thom spectrum functor $\T$ preserves colimits by \cite[Example 2.5]{HY17}. 
\end{rem}

The following lemma is a special case of \cite[Theorem 2.2.2.4]{HA}, for which we provide a more concise proof.

\begin{lem}\label{lem:sym}
	Let $Z$ be an $\E $-space, then the slice-category $\Ss_{/Z}$ has a symmetric monoidal structure which is given by the formula
	\begin{equation}\label{eq:prod/Z}
		(X \to Z)\times (Y \to Z) = (X\times Y \to Z\times Z \xrightarrow{\mu_Z} Z).
	\end{equation}
\end{lem}

\begin{proof}
	We model the symmetric monoidal structure on $\Ss_{/Z}$ by a simplicial commutative monoid. Given two $n$-simplices $f,g\colon (\Delta^n)^{\triangleright} \to \Ss$ of $\Ss_{/Z}$, the map
	\[m^{\prime}\colon (\Delta^n)^{\triangleright} \xrightarrow{\Delta} (\Delta^n)^{\triangleright} \times (\Delta^n)^{\triangleright} \xrightarrow{f\times g} \Ss \times \Ss \xrightarrow{\times} \Ss\]
	maps the $n$-simplex $\Delta^n$ to $f|_{\Delta^n} \times g|_{\Delta^n}$ and the cone to $Z \times Z$. Therefore, we have the following diagram	
\[
\xymatrix{
\Delta^0 \ar[r]^-{i} \ar[d]_-{d^1}
&(\Delta^n)^{\triangleright} \ar[d] \ar@{}[ld]|{\ulcorner} \ar@/^10pt/[rdd]^{m^{\prime}} \ar@{}@<0.5ex>[rdd]&\\
\Delta^1 \ar[r] \ar@/_10pt/[rrd]_{\mu_Z}\ar@{}@<-0.5ex>[rrd]
&\Delta^n \star \Delta^1 \ar@{>}[rd]_{m} &\\
&&\Ss
}\]
The symmetric monoidal structure takes $f$ and $g$ to
\[(\Delta^n)^{\triangleright} \cong \Delta^n \star \Delta^0 \xrightarrow{id \star d^0} \Delta^n \star \Delta^1 \xrightarrow{m} \Ss. \qedhere\]
\end{proof}

\begin{rem}\label{rm:rig}
	The simplicial category $(\mathrm{Set}_{\Delta})_{/Z}$ admits a model structure whose (co)fibrations and weak equivalences are precisely the morphisms whose images under the forgetful functor are (co)fibrations and weak equivalences, respectively. A direct verification shows that Formula \ref{eq:prod/Z} endows $(\mathrm{Set}_{\Delta})_{/Z}$ with a simplicial symmetric monoidal structure. Consequently, we can also model the symmetric monoidal structure on $\Ss_{/Z}$ by the simplicial symmetric monoidal structure on $(\mathrm{Set}_{\Delta})_{/Z}$, by \cite[Proposition 4.1.7.10]{HA}. Similarly, by \cite[Theorem B.5]{KP25-2} there is an equivalence of symmetric monoidal categories 
	\[\mathrm{Fun}(X^{op},\Ss)\simeq \mathrm{N}(\mathrm{Fun}(\mathfrak{C}[X^{op}],\mathrm{Set}_{\Delta})^c)[\mathcal{W}_{proj}^{-1}]\]
	with both categories equipped with the Day convolution, where $\mathrm{N}$ is the simplicial nerve and $\mathcal{W}_{proj}^{-1}$ is the collection of projective weak equivalences.
\end{rem}

The following proposition is well-known, but for the sake of completeness, we include its proof here.

\begin{prop}\label{prop:sym str}
	The Thom spectrum functor is symmetric monoidal when $\Ss_{/Pic(R)}$ is equipped with the aforementioned symmetric monoidal structure.
\end{prop}

\begin{proof}
Generally, let $X$ be an $\E $-space, by the straightening and unstraightening constructions, there is an equivalence of categories 
    \[
        \begin{tikzcd}
            \Ss_{/X} \arrow[r, shift left, "St_X"] & \mathrm{Fun}(X^{\mathrm{op}}, \Ss) \arrow[l, shift left, "Un_X"].
        \end{tikzcd}
    \]
We claim that this equivalence is symmetric monoidal, where the symmetric monoidal structure on $\mathrm{Fun}(X^{\mathrm{op}},\Ss)$ is given by the Day convolution. Indeed, by Lemma \ref{lem:sym}, Remark \ref{rm:rig} and the definition of the Day convolution, it suffices to prove the following diagram commutes
\[
\xymatrix{
\Ss_{/X} \times \Ss_{/X} \ar[r]^-{\times} \ar[d]_-{St_X \times St_X}
&\Ss_{/X \times X} \ar[d]_-{St_{X \times X}} \ar[r]^-{\mu_!}
&\Ss_{/X} \ar[d]_-{St_{X}} &\\
\mathrm{Fun}(X^{\mathrm{op}}, \Ss) \times \mathrm{Fun}(X^{\mathrm{op}}, \Ss) \ar[r]_-{\times}
&\mathrm{Fun}(X^{\mathrm{op}} \times X^{\mathrm{op}}, \Ss) \ar[r]_-{\mu_!}
&\mathrm{Fun}(X^{\mathrm{op}}, \Ss)
}\]
where the lower left map is the point-wise tensor product in $\Ss$. The left diagram commutes by \cite[Remark 2.2.2.12]{HA}. Noting that $St_X\circ \mu_!\simeq St_{\mu}$, it follows from  \cite[Proposition 2.2.1.1]{HA} that the right one commutes. 

In particular, taking $X=Pic(R)$ and noting that $Pic(R) \xrightarrow{i} \M$ is symmetric monoidal, we obtain a symmetric monoidal functor $\Ss_{/Pic(R)} \xrightarrow{St_X} \mathrm{Fun}(Pic(R)^{\mathrm{op}}, \Ss) \to \M.$ It remains to show that this functor is exactly the Thom spectrum functor. Since the functor $\T$ preserves colimits, by the universal property of the left Kan extension, it suffices to prove that for any object $L\colon \Delta^0 \to Pic(R)$ of $Pic(R)$, the Thom spectrum $\T(Un_X(L))$ is $L \in \M$. This is clear, since $Un_X(L)= L \colon \Delta^0 \to Pic(R).$
\end{proof}

From now on, we only consider connected spaces and based maps in $\Ss_{/Pic(R)}$ where the base point of $Pic(R)$ is $R$. In this case, maps factor through $B\G$, the connected component of $Pic(R)$ containing the unit $R$. Note that $B\G$ itself is an $\E$-space. Consequently, we pass to the slice category $\Ss_{/B\G}$, and with this replacement, all the above discussions apply. Furthermore, the functors $St$ and $Un$ correspond to taking the fiber and the colimit, respectively. Therefore, we have the following commutative diagrams

\begin{center}
\begin{minipage}{0.53\textwidth}
\centering
\[
\xymatrix{
\Ss_{/*} \ar[r]^-{e_!} \ar[d]_-{\mathrm{fib}}
&\Ss_{/B\G} \ar[d]^-{\mathrm{fib}} &\\
\mathrm{Fun}(*,\Ss)  \ar[r]_-{e_!}
&\mathrm{Fun}(B\G^{\mathrm{op}}, \Ss),
}\]
\end{minipage}
\hfill
\begin{minipage}{0.46\textwidth}
\centering
\[
\xymatrix{
\Ss_{/*} 
&\Ss_{/B\G} \ar[l]_-{e^*} &\\
\mathrm{Fun}(*,\Ss) \ar[u]^-{\colim} 
&\mathrm{Fun}(B\G^{\mathrm{op}}, \Ss). \ar[u]_-{\colim} \ar[l]^-{e^*}
}\]
\end{minipage}
\end{center}
Here, $e\colon * \to B\G$ is the basepoint, and functors $e_!$ above and below denote the base change and the left Kan extension, respectively. Moreover, maps in the right diagram are right adjoints of the corresponding maps in the left diagram.

\begin{lem}\label{lem:proj}
	Let $X \in \mathrm{Fun}(B\G^{\mathrm{op}}, \Ss)$ and $Y \in \Ss$, there is an equivalence
	\[X\otimes e_!Y\simeq X\times Y,\]
	where $X\times Y$ is the colimit of the constant functor $Y \to \mathrm{Fun}(B\G^{\mathrm{op}}, \Ss)$ that takes the value $X$.
\end{lem}
\begin{proof}
	Since colimits in presheaf categories are computed point-wise, the underlying space of $X \times Y$ is exactly $\colim\limits_Y X \simeq \colim\limits_Y X \times * \simeq X \times Y \in \Ss.$ Furthermore, for an object $Y\in \Ss$ and a map $f:X\to B\G\in \Ss_{/B\G}$, we have $X\times e_! Y\simeq \colim\limits_Y X$ by the definition of the tensor product of the category $\Ss_{/B\G}$, since colimits in $\Ss_{/B\G}$ are also computed in $\Ss$.  Consequently, we have the following equivalences
	\[
		\begin{split}
			X \otimes e_! Y & \simeq \ \mathrm{fib}\ (\colim\limits_{B\G^{\mathrm{op}}} \ X \times \colim\limits_{B\G^{\mathrm{op}}} e_! Y)\\
			& \simeq \ \mathrm{fib}\ (\colim\limits_{B\G^{\mathrm{op}}} X \times e_! \colim\limits_{\Delta^0} Y)\\
			& \simeq \ \mathrm{fib}\ (\colim\limits_{Y} \colim\limits_{B\G^{\mathrm{op}}} X)\\
			& \simeq \ \colim\limits_{Y} (\mathrm{fib}\ (\colim\limits_{B\G^{\mathrm{op}}} X)) \\
			& \simeq X\times Y.
		\end{split}
	\]
	We justify the sequence of equivalences as follows: the first follows from the fact that functors $\mathrm{fib}$ and $\colim$ are mutual inverses and that $\colim$ is symmetric monoidal, as proven in Proposition \ref{prop:sym str}. The second relies on the assertion made prior to the lemma. Other equivalences are immediate.
\end{proof}

We are now ready to prove the main result of this section:

\begin{thm}[Theorem \ref{thm:comodule}]
	Let $R$ be an $\E$-ring spectrum and $X$ be a connected space. Given a map of spaces $f\colon X\to B\G$, the Thom spectrum $\T(R_X)\simeq R[X]$ is an $\E $-R-coalgebra, and $\T f$ is a $\T(R_X)$-comodule, where $R_X$ is the constant local system taking the value $R$.
\end{thm}
\begin{proof}
Being a left Kan extension, the functor $e_!$ is symmetric monoidal. Furthermore, Proposition \ref{prop:sym str} implies that $\T$ is symmetric monoidal. Consequently, $\T(R_X)=\T(e_!X)$ is an $\E $-R-coalgebra, since any space is an $\E $-coalgebra in $\Ss$. For the same reason, it remains to prove that $f$ is an $e_!X$-comodule in the category $\Ss_{/B\G}$. Equivalently, we prove that $Y=\mathrm{fib}(f)$ is an $e_!X$-comodule in the category $\mathrm{Fun}(B\G^{\mathrm{op}}, \Ss)$. Using Lemma \ref{lem:proj} repeatedly, we have
\[
Y \otimes \underbrace{e_!X \otimes e_!X \otimes \cdots \otimes e_!X}_{n} 
\simeq 
Y \times \underbrace{X \times X \times \cdots \times X}_{n},
\]
where the action of $\G$ on $Y \times X \times \cdots \times X$ is given by the canonical action on $Y=\mathrm{fib}(f)$ and the trivial action on $X\simeq \colim\limits_{B\G^{\mathrm{op}}} Y.$ The problem has now been reduced to verifying that $Y$ is a $\colim\limits_{B\G^{\mathrm{op}}} Y$-comodule in $\Ss$. This comodule structure is induced by the unit map $Y\to p^*p_!Y$, where $p$ is the unique map $B\G^{op}\to *$.
\end{proof}

Using the fact that $\colim\limits_{B\G^{\mathrm{op}}}$ is oplax monoidal when $\mathrm{Fun}(B\G^{\mathrm{op}}, \Ss)$ is equipped with the point-wise tensor product, \cite[Corollary 4.14]{Bea23} provides an alternative construction of an $R[X]$-comodule structure on $\T f$. We show that this structure agrees with ours. For clarity in the proof, we write the comodule structure in \cite[Corollary 4.14]{Bea23} as a $\T(R_X)$-comodule to distinguish it from the one in Theorem \ref{thm:comodule}.

\begin{prop}\label{prop:compare}
	The $R[X]$-comodule structure on $\T f$ in Theorem \ref{thm:comodule} coincides with the $\T(R_X)$-comodule structure on $\T f$ defined in \cite[Corollary 4.14]{Bea23}.
\end{prop}
\begin{proof}
	As shown in the proof of Theorem 4.15 of \cite{Bea23}, the coaction of $\T(R_X)$ on $\T f$ is given by applying the left Kan extension along $p\colon X\to *$ to the unit
	\[f\to p^*p_!f\simeq p^*p_!f\otimes_R R_X,\]
	where $\otimes_R$ is the point-wise tensor product. On the other hand, the coaction $\epsilon\colon \T f\to \T f\otimes_R R[X]$ in Theorem \ref{thm:comodule} is induced by applying the Thom spectrum functor to the following diagram
	\begin{equation}\label{eq:comod}
    \xymatrix{
    X \ar[r]^-{\Delta} \ar[dr]_-{f}
    &X\times X \ar[d]^-{(f,R_X)}
    \\
    &Pic(R),
    }
	\end{equation}
    Therefore, for any $R$-module $M$, there is a commutative diagram
    \[
    \xymatrix{
    \mathrm{Map}_R(\T f\otimes_R R[X], M) \ar[r]^-{\epsilon^*} \ar[d]_-{\simeq}
    &\mathrm{Map}_R(\T f, M) \ar[d]^-{\simeq} &\\
    \mathrm{Map}_{\M ^{X\times X}}((f,R_X), p^*M)  \ar[r]_-{\Delta^*}    &\mathrm{Map}_{\M ^X}(f, p^*M).
    }\]
   Furthermore, there is an equivalence 
   \[\colim\limits_{X\times X}\ (f,R_X)\xrightarrow{\simeq} \colim\limits_{X}[ p^*(\colim\limits_{X}f)],\]
     by \cite[Example 2.5]{HY17}. Applying $\mathrm{Map}_R(-,M)$ to this equivalence and composing with $\Delta^*$, we obtain a map
     \[\mathrm{Map}_{\M ^X}(p^*{\T f}, p^*M)\to \mathrm{Map}_{\M ^X}(f,p^*M).\]
     Unwinding definitions, we see that this map is induced by the unit $f\to p^*p_!f$.
\end{proof}

We end this section by proving that $\T f$ can be expressed as a relative tensor product for a map $f:X\to B\G$ with $X$ connected, which will be used in Section \ref{sec4}. Equation \ref{eq:presheaf} below is a special case of \cite[Lemma 4.47]{CCRY25} with an identical proof. We include this proof as part of the proof of Proposition \ref{prop:map}, as it is required for establishing our subsequent results. The following proposition will be used to compute the structural maps of $\mathrm{coTHH}^R(\T f;R[X])$.

\begin{prop}\label{prop:map}
	Let $G$ be a group-like $\mathbb{A}_{\infty}$-space and $M,N$ be objects of $\M ^{BG}$. The mapping space $\mathrm{Map}_{\M ^{BG}}(M,N)$ is equivalent to the totalization of the following cosimplicial space
	\[\mathrm{Map}_R(M,N)\rightrightarrows\mathrm{Map}_R(M\otimes G,N)\threerightarrows \mathrm{Map}_R(M\otimes (G\times G),N)\fourrightarrows\cdots,\]
	where the cosimplicial structure is induced by actions of $G$ on $M$ and $N$ (see the proof).
\end{prop} 
\begin{proof}
	Taking $f=e\colon *\to BG$ in \cite[Corollary 4.12]{CCRY25}, we see that there is an equivalence of left $\M $-categories
	\[\M ^{BG}\simeq \M[BG]= \M\otimes \Ss[BG]\]
	under $\M $. Furthermore, $\M[BG]$ is the image of $G$ under the composition of functors $\mathrm{Alg}^{gp}(\Ss)\xrightarrow{B}\Ss_{*/}\xrightarrow{\M[-]}\mathrm{Mod}({\M })_{\M /}$, where $\mathrm{Mod}({\M })$ is the category of  $\M$-modules in the category $\mathrm{Pr}^L$ (note that since $\M$ is symmetric monoidal, there is an equivalence $\mathrm{LMod}(\M)\simeq \mathrm{Mod}(\M)$). As discussed at the beginning of Section \ref{sec2}, both functors admit right adjoints, and therefore, their composition has a right adjoint
	\[\mathrm{Mod}({\M })_{\M /}\xrightarrow{(-)^\sim}\Ss_{*/}\xrightarrow{\Omega} \mathrm{Alg}^{gp}(\Ss).\]
	This functor takes $\mathcal{C}\in \mathrm{Mod}({\M })_{\M /}$ to the subspace of $\mathrm{Map}_{\mathcal{C}}(1_{\mathcal{C}},1_{\mathcal{C}})$ consisting of invertible maps. Hence, we can also write this right adjoint as the composition
	\[\mathrm{Mod}({\M })_{\M /}\xrightarrow{\mathrm{End}(1_{-})}\mathrm{Alg}(\mathrm{Mod_R})\xrightarrow{(-)^{\times}}\mathrm{Alg}^{gp}(\Ss).\]
	Again, the second functor is right adjoint to $R[-]$, while by \cite[Proposition 4.8.11]{HA} and its remark the first functor is right adjoint to $\mathrm{RMod}(-)$. By the uniqueness of left adjoints, we obtain an equivalence
	\begin{equation}\label{eq:presheaf}
			F\colon\mathrm{Mod_R}^{BG} \simeq \M[BG]\xrightarrow{\simeq} \mathrm{RMod}_{R[G]}(\M ).
	\end{equation}
	Next, we need to explain the meaning of the $G$-action on $M\in\M^{BG}$. In other words, we must prove that $UF(M)=M$, where $U$ is the forgetful functor. Indeed, we have the following commutative diagrams
	\begin{center}
\begin{minipage}{0.53\textwidth}
\centering
\[
\xymatrix{
\M \ar[r]^-{\simeq} \ar[d]_-{e_!}
&\mathrm{RMod}_{R}(\M ) \ar[d]^-{-\otimes_{R} R[G]} &\\
\M[BG]  \ar[r]_-{\simeq}
&\mathrm{RMod}_{R[G]}(\M ),
}\]
\end{minipage}
\hfill
\begin{minipage}{0.46\textwidth}
\centering
\[
\xymatrix{
\M \ar[r]^-{\simeq} 
&\mathrm{RMod}_{R}(\M ) &\\
\M[BG]  \ar[r]_-{\simeq} \ar[u]^-{e^*}
&\mathrm{RMod}_{R[G]}(\M ) \ar[u]_-{U}.
}\]
\end{minipage}
\end{center}
The left diagram commutes by the naturality of the construction above, and the right one is obtained from the left diagram by passing to right adjoints. Note that equivalences $\M \to \mathrm{RMod}_R(\M)$ in both diagrams are identity maps. Furthermore, by \cite[Corollary 4.12]{CCRY25}, the diagram
\[
\xymatrix{
\M[BG] \ar[r]^-{\simeq} \ar[d]_-{e^*}
&\M^{BG} \ar[d]^-{e^*} &\\
\M  \ar[r]_-{\simeq}
&\M,
}\]
also commutes. Putting these diagrams together yields the desired result.
	
Now as a right $R[G]$-module in the category $\M $, $M$ is equivalent to the geometric realization of the bar construction $\mathrm{Bar}_{\bullet}^R(M,R[G],R[G])$ in $\M $. Therefore, the mapping space $\mathrm{Map}_{\mathrm{RMod}_{R[G]}(\M )}(M,N)$ is equivalent to the totalization of the cosimplicial space
	\[\mathrm{Map}_{\mathrm{RMod}_{R[G]}(\M )}(\mathrm{Bar}_{\bullet}^R(M,R[G],R[G]),N).\]
	By the base change, this is equivalent to
	\[\mathrm{Tot}(\mathrm{Map}_R(M,N)\rightrightarrows\mathrm{Map}_R(M\otimes G,N)\threerightarrows \mathrm{Map}_R(M\otimes (G\times G),N)\fourrightarrows\cdots),\]
	which yields the desired conclusion.
\end{proof}

Using the identification in Proposition \ref{prop:map}, we can express $\T f$ in the form of a bar construction.

\begin{thm}\label{thm:Thf}
	Let $X$ be a connected space and $f$ be a map from $X$ to $B\G$, then there is an equivalence of $R$-modules
	\[\T f\simeq R\otimes_{R[G]}R,\]
	here $G$ is the loop space of $X$. Furthermore, the $R[G]$-module structure on the left copy of $R$ arises from the action of $G$ on $R$, while that on the right copy comes from the map of algebras induced by the map $p\colon G\to *$.
\end{thm}
\begin{proof}
	Since $\T f$ is defined as a colimit, for any $R$-module $M$ there is an equivalence of spaces
	\[\mathrm{Map}_R(\T f,M)\simeq \mathrm{Map}_{\M^X}(f,p^*M).\]
	As in the proof of the preceding proposition, using the naturality and \cite[Corollary 4.12]{CCRY25}, we have the following commutative diagram
\[
\xymatrix{
\M \ar[r]^-{\simeq} \ar[d]_-{p^*}
&\M \ar[r]^-{\simeq} \ar[d]_-{p^*}
&\mathrm{RMod}_R(\M) \ar[d]^-{U} &\\
\M^X \ar[r]_-{\simeq}
&\M[X]  \ar[r]_-{\simeq}
&\mathrm{RMod}_{R[\Omega X]}(\M).
}\]
Therefore, under the equivalence $\M^X\simeq \mathrm{RMod}_{R[\Omega X]}(\M)$, the image of $p^*M\in\M^X$ is $U(M)$ and $U$ is right adjoint to the functor $-\otimes_{R[\Omega X]}R$. Hence, by the proof of Proposition \ref{prop:map}, we obtain that
	\[\mathrm{Map}_{\M^X}(f,p^*M)\simeq \mathrm{Map}_{\mathrm{RMod}_{R[G]}(\M)}(R,UM)\simeq \mathrm{Map}_R(R\otimes_{R[G]}R,M).\]
	Now the result follows from the Yoneda Lemma.
\end{proof}

\begin{rem}
	Let $\mathcal{C}$ be a presentable symmetric monoidal category, $A_i$ be algebras in $\mathcal{C}$, $M_i$ be right $A_i$-modules and $N_i$ be left $A_i$-modules for $i=1,2$, we have
	\[|\mathrm{Bar}_{\bullet}^{\mathcal{C}}(M_1,A_1,N_1)|\otimes|\mathrm{Bar}_{\bullet}^{\mathcal{C}}(M_2,A_2,N_2)|\simeq|\mathrm{Bar}_{\bullet}^{\mathcal{C}}(M_1\otimes M_2,A_1\otimes A_2,N_1\otimes N_2)|,\]
	since $\Delta^{op}$ is sifted. In particular, given two pointed maps $f\colon X\to Pic(R)$ and $g\colon Y\to Pic(R)$  with $X,Y$ connected, we obtain
	\begin{equation*}
		\begin{split}
			\T f\otimes_R\T g & \simeq (R\otimes_{R[\Omega X]}R)\otimes_R (R\otimes_{R[\Omega Y]}R)\\
			& \simeq (R\otimes_R R)\otimes_{(R[\Omega X]\otimes_R R[\Omega Y])}(R\otimes_R R)\\
			& \simeq R\otimes_{R[\Omega (X\times Y)]}R\\
			& \simeq \T (f,g).
		\end{split}
	\end{equation*}
	Thus, we have established an alternative proof of the symmetric monoidality of the Thom spectrum functor.
\end{rem}

\begin{rem}\label{rem:ind}
	Let $A$ be an $\E$-$R$-algebra, then the functor $\mathrm{Ind}_R^A:\M\to \mathrm{Mod}_A$ is a map of commutative algebras in the category $\mathrm{Pr}^L$ and induces an adjunction
	\begin{equation}\label{fun:ad}
		\mathrm{Ind}^{\mathrm{Mod}_A}_{\M}: \mathrm{Mod}(\M)_{\M/} \rightleftarrows \mathrm{Mod}(\mathrm{Mod}_A)_{\mathrm{Mod}_A/}:U
	\end{equation}
	 Composing this functor with the functor $\mathrm{Alg}^{gp}(\Ss)\xrightarrow{B}\Ss_{*/}\xrightarrow{\M[-]}\mathrm{Mod}({\M })_{\M /}$ in the proof of Proposition \ref{prop:map}, the same argument shows that there is a commutative diagram
	 \[
       \xymatrix{
       \M[BG] \ar[r]^-{\simeq} \ar[d]_-{\mathrm{Ind}_R^A}
       &\mathrm{RMod}_{R[G]}(\M) \ar[d]^-{\mathrm{Ind}_R^A} &\\
       \mathrm{Mod}_A[BG]  \ar[r]_-{\simeq}
       &\mathrm{RMod}_{A[G]}(\M)
       }\]
    by considering the unit of the adjunction \ref{fun:ad}. In particular, we see from the proof of Theorem \ref{thm:Thf} that 
    \[\mathrm{Ind}_R^A(\T f)\simeq A\otimes_{A[G]}A.\]
    In other words, the equivalence of Theorem \ref{thm:Thf} is natural with respect to the base change. Alternatively, this equivalence can also be proven by employing a series of adjoint functors as in the proof of Theorem \ref{thm:Thf}.
\end{rem}

\section{Topological coHochschild homology of Thom spectra}\label{sec3}
Let $R$ be an $\E $-ring and $f$ be a map from $X$ to $Pic(R)$ with $X$ simply connected. In this section, we study the topological coHochschild homology of $R[X]$ with coefficient $\T f$. Note that there is an equivalence $X\simeq B\Omega X$, since $X$ is simply connected. Hence, we may always assume that the $0$-skeleton of $X$ consists of a single point. We begin by recalling the definition of topological coHochschild homology. Recall that a coalgebra $C$ and a $C$-bicomodule in a symmetric monoidal category $\mathcal{C}$ can be interpreted as an algebra and a bimodule in the opposite symmetric monoidal category $\mathcal{C}^{op}$ (see \cite[Remark 2.4.2.7]{HA} or \cite[Definition 2.1]{Pe22}). Therefore, we can regard the pair $(M,C)$ as an object of $\mathrm{CycBMod}_1(\mathcal{C}^{op})$ (see \cite[Definition 6.15]{KMN23}). 

\begin{Def}\cite[Definition 6.25]{KMN23}
	Let $\mathcal{C}$ be a symmetric monoidal category, there is a functor 
	\[\mathrm{coTHH}^\mathcal{C}_*: \mathrm{CycBMod}_1(\mathcal{C}^{op})^{op} \hookrightarrow \mathrm{Fun}^{\otimes}((\mathrm{CycBMod}_1^{\otimes})_{act}, \mathcal{C}^{op})^{op}\xrightarrow{\mathrm{Cut}_{\Lambda}(-\times [1])^*}\mathrm{Fun}(\Delta,\mathcal{C}).\]
	Topological coHochschild homology of $C$ with coefficient $M$ is defined to be 
	\[\mathrm{coTHH}^\mathcal{C}(M;C):=\mathrm{Tot}(\mathrm{coTHH}^\mathcal{C}_*(M;C)).\]
	When $M=C$, $\mathrm{coTHH}^\mathcal{C}(M;C)$ is precisely the topological coHochschild homology of $C$ which is denoted by $\mathrm{coTHH}^{\mathcal{C}}(C)$. 
\end{Def}    
    Informally, $\mathrm{coTHH}^\mathcal{C}(M;C)$ is the totalization of the following cosimplicial $R$-module
\[
M \;
\mathrel{\substack{\displaystyle\longrightarrow\\[-0.6ex]\displaystyle\longleftarrow\\[-0.6ex]\displaystyle\longrightarrow}}
\;
M \otimes_R C \;
\mathrel{\substack{\displaystyle\longrightarrow\\[-0.6ex]\displaystyle\longleftarrow\\[-0.6ex]\displaystyle\longrightarrow\\[-0.6ex]\displaystyle\longleftarrow\\[-0.6ex]\displaystyle\longrightarrow}}
\;
M\otimes_R C\otimes_R C \; \mathrel{\substack{\displaystyle\longrightarrow\\[-0.6ex]\displaystyle\longleftarrow\\[-0.6ex]\displaystyle\longrightarrow\\[-0.6ex]\displaystyle\longleftarrow\\[-0.6ex]\displaystyle\longrightarrow\\[-0.6ex]\displaystyle\longleftarrow\\[-0.6ex]\displaystyle\longrightarrow}} \; \cdots,
\]
whose coface maps are induced by the right and left coactions of $C$ on $M$, the comultiplication of $C$ and twist maps. For an $\E$-ring $R$, an $R$-coalgebra $C$ and a $C$-bicomodule $M$, we use the shorthand notation $\mathrm{coTHH}^R(M;C)$ to denote $\mathrm{coTHH}^{\M}(M;C).$

\begin{exmp}\label{exmp:over BG}
    Note that any space $X$ is a cocommutative coalgebra in $\Ss$. Unwinding the definitions, there is an equivalence of cosimplicial space
    \[\mathrm{coTHH}^{\Ss}_*(X)\simeq \mathrm{Map}(S^1_{\bullet}, X),\]
    where we consider $S^1_{\bullet}=\Delta^1/\partial\Delta^1$ as a simplicial space which is constant in each degree.
    Consequently, there is an equivalence
    \[\mathrm{coTHH}^{\Ss}(X)\simeq\mathrm{Tot}\mathrm{Map}(S^1_{\bullet}, X)\simeq \mathrm{Map}(|S^1_{\bullet}|,X)=LX,\]
    where $LX$ is the free loop space of $X$. Furthermore, given a map $f:X\to B\G$, we see that $f$ is a comodule over $R_X$ by the proof of Theorem \ref{thm:comodule}. Since the forgetful functor $\Ss_{/B\G} \to S$ preserves contractible limits, the underlying space of $\mathrm{coTHH}^{\Ss_{/B\G}}(f,R_X)$ is $\mathrm{coTHH}^{\Ss}(X)\simeq LX$. Consequently, $\mathrm{coTHH}^{\Ss_{/B\G}}(f,R_X)$ is the map
	\[\mathrm{coTHH}^{\Ss_{/B\G}}(f;R_X):LX\xrightarrow{ev_0}X\xrightarrow{f} B\G.\]
\end{exmp}

In \cite{Mal17}, it is shown that 
\[\mathrm{coTHH}^{\mathbb{S}}(\mathbb{S}[X])=\mathrm{coTHH}^{\mathbb{S}}(\mathbb{S}[X];\mathbb{S}[X])\simeq \mathbb{S}[LX]\]
for a simply connected space $X$. The next example shows that $\mathrm{coTHH}^{\mathbb{S}}(\mathbb{S};\mathbb{S}[X])$ is equivalent to the suspension spectrum of $\Omega X$ for a pointed simply connected space $X$.
\begin{exmp}\label{exmp:Omega}
	Let $X$ be a pointed simply connected space with the base point $e$, then $e:*\to X$ is a map of $\E $-coalgebras in the category $\Ss$. Therefore, $e: \mathbb{S} \to \mathbb{S}[X]$ is a map of $\E $-coalgebras in the category $\mathrm{Sp}$, since $\mathbb{S}[-]$ is symmetric monoidal. In particular, $\mathbb{S}$ is a $\mathbb{S}[X]$-comodule. Setting $M=D=\mathbb{S}$ and $C=\mathbb{S}[X]$ in \cite[Proposition 4.6]{Z25}, we have
	\[\mathrm{coTHH}^{\mathbb{S}}(\mathbb{S}; \mathbb{S}[X])\simeq \mathbb{S}\square_{\mathbb{S}[X]}\mathbb{S}\]
	In other words, $\mathrm{coTHH}^{\mathbb{S}}(\mathbb{S};\mathbb{S}[X])$ is equivalent to the cobar construction of $\mathbb{S}[X]$ which is equivalent to $\mathbb{S}[\Omega X]$ by \cite[Proposition 3.15]{HS21}. More generally, it can be deduced from \cite[Lemma 3.2]{Z25} or Proposition \ref{prop:base change} below that there is an equivalence $\mathrm{coTHH}^R(R; R[X])\simeq R[\Omega X]$ for any connective $\E$-ring $R$.
\end{exmp}

Let $A$ be an $\E$-ring. Just as we can use the B\"okstedt spectral sequence to compute the $A$-homology of $\mathrm{THH}(R)$, in \cite{BGT+18}, the authors introduced the coB\"okstedt spectral sequence to compute the $A$-homology of $\mathrm{coTHH}^R(C)$. Similarly, this spectral sequence can also be applied to compute $\mathrm{coTHH}^R(M;C)$. Consequently, we begin by analyzing the $A$-homology of $\mathrm{coTHH}^R(\T f;R[X])$ for a suitable $\E$-ring $A$. Recall that an $\mathbb{E}_1$-$R$-coalgebra $C$ is coaugmented if there is a map of $\mathbb{E}_1$-$R$-coalgebras $R\to C$. Its coaugmentation coideal is the fiber of the counit.

\begin{prop}\label{prop:base change}
   Let $R$ be a connective $\E$-ring, $A$ be a connective $\E$-R-algebra, $C$ be a coaugmented $\mathbb{E}_1$-$R$-coalgebra whose homotopy groups of coaugmentation coideal $\bar{C}$ vanish in degree less than $2$ and $M$ be a connective $C$-bimodule, then there is an equivalence of $A$-modules
       \[\mathrm{Ind}_R^A(\mathrm{coTHH}^R(M;C))\simeq \mathrm{coTHH}^A(\mathrm{Ind}_R^A(M),\mathrm{Ind}_R^A(C)).\]
\end{prop}
\begin{proof}
    By \cite[Theorem 1.1]{Z25}, there is an equivalence 
    \[\mathrm{coTHH}^R(M;C)\simeq M\square_{C\otimes_R C^{op}} C.\]
    Furthermore, as in the proof of \cite[Lemma 3.2]{Z25}, there is a fiber sequence
	\begin{equation*}
		\mathrm{fib}(f_n)\to \mathrm{fib}(f_{n-1})\xrightarrow{\rho} \mathrm{fib}(f_{n-1})\otimes_R (C\otimes_R C^{op}),
	\end{equation*}
	here $\mathrm{fib}(f_i)$ is the fiber of the natural map $M\to \mathrm{Tot}^i(\Omega^*_R(M,C\otimes_R C^{op},C))$ for $i\geq 0$, and the fiber of $i_n^*:\mathrm{Tot}^{n+1}(\Omega^*_R(M,C\otimes_R C^{op},C))\to \mathrm{Tot}^n(\Omega^*_R(M,C\otimes_R C^{op},C))$ is equivalent to $\mathrm{fib}(f_n)\otimes_R C$. Note that $\rho$ admits a left inverse $id\otimes_R\epsilon$, and $(id\otimes_R\epsilon)_*$ is an isomorphism when restricted to $\pi_*({\mathrm{fib}(f_{n-1})})$, where $\epsilon$ is the counit of $C\otimes_R C^{op}$. Therefore, $\rho_k$ and $\rho_{k+1}$ are isomorphisms because $\overline{C\otimes_R C^{op}}$ is $2$-connective, where $k$ is the connectivity of $\mathrm{fib}(f_{n-1})$. Hence $\mathrm{fib}(f_n)\otimes_R C$ is $n$-connective by induction. Now since $R$ and $A$ are connective, the functor $\mathrm{Ind}_R^A$ preserves connectivity. Therefore, the fiber $F^n$ of the map
	\[\mathrm{Ind}_R^A(\mathrm{coTHH}^R(M;C))\to \mathrm{Ind}_R^A[\mathrm{Tot}^n\Omega_R^*(M,C\otimes_R C^{op},C))]\simeq \mathrm{Tot}^n[\mathrm{Ind}_R^A\Omega_R^*(M,C\otimes_R C^{op},C))]\]
	is $(n-1)$-connective by the Milnor sequence. Here we use the fact that $\mathrm{Tot}^i$ is equivalent to a finite limit, and hence commutes with $\mathrm{Ind}_R^A$. It follows that $\lim\limits_nF^n\simeq 0$ and
	\[\mathrm{Ind}_R^A(\mathrm{coTHH}^R(M;C))\simeq \mathrm{Tot}[\mathrm{Ind}_R^A\Omega_R^*(M,C\otimes_R C^{op},C))]\simeq \mathrm{coTHH}^A(\mathrm{Ind}_R^A(M),\mathrm{Ind}_R^A(C)),\]
	since the functor $\mathrm{Ind}_R^A$ is symmetric monoidal.
\end{proof}

\begin{rem}
	For a general cobar complex $\Omega^*_R(M,C,N)$, the fiber sequence in the proof of \cite[Lemma 3.2]{Z25} also implies that $\mathrm{fib}(f_n)\simeq \Omega (\mathrm{fib}(f_{n-1})\otimes_R\bar{C})$, since $\rho$ admits a left inverse. Hence, we deduce that
	\begin{equation}\label{eq:fib}
		\mathrm{fib}(i^*_n)\simeq \Omega^n(M\otimes_R\bar{C}^{\otimes_R^{n+1}}\otimes_R N).
	\end{equation}
\end{rem}

\begin{thm}[Theorem \ref{thm:orient}]
	Let $R$ be a connective $\E$-ring spectrum and $X$ be a simply connected space. Given a connective $\E$-$R$-algebra $A$ and a map $f\colon X\to B\G$, there is an equivalence
	\[A\otimes_R\mathrm{coTHH}^R(\T f;R[X])\simeq A\otimes\mathrm{coTHH}(\mathbb{S}[X])\]
	if $f$ is $A$-oriented. Here we abbreviate $\mathrm{coTHH}^{\mathbb{S}}(\mathbb{S}[X])$ to $\mathrm{coTHH}(\mathbb{S}[X])$.
\end{thm}
\begin{proof}
	Recall from \cite[Definition 3.14]{AB19} that an $A$-orientation of $f$ is a lift $\tilde{f}$ of $f$:
	\[
    \xymatrix{
    &B(R,A) \ar[d]
    &\\
    X \ar[r]_-{f} \ar@{-->}[ru]^-{\tilde{f}}   
    &Pic(R),
    }\]
    where $B(R,A)$ is the full symmetric monoidal subgroupoid of $Pic(R)_{/A}$ consisting of maps of $R$-module $h:M\to A$ such that the adjoint $h^{\dagger}:\mathrm{Ind}_R^A(M)\to A$ is an equivalence. Since $\mathrm{Ind}_R^A(R)\simeq A$, the constant functor $R_X:X\to Pic(R)$ admits a lift $\tilde{R}_X$ to $B(R,A)$, and hence we have the following commutative diagram
    \[
    \xymatrix@C=1.7cm@R=1cm{
    &B(R,A) \times \cdots \times B(R,A) \ar[r]^-{\otimes} \ar[d]
    &B(R,A) \ar[d] &\\
    X \times \cdots \times X \ar[r]_-{(f,R_X,\cdots,R_X)} \ar[ru]^-{(\tilde{f},\tilde{R}_X,\cdots,\tilde{R}_X)}
    &Pic(R) \times \cdots \times Pic(R)  \ar[r]_-{\otimes}
    &Pic(R).
    }\]
    In other words, maps $(f,R_X,\cdots,R_X)$ are also $A$-oriented. As observed in the proof of \cite[Proposition 3.16]{AB19}, $B(R,A)$ is the pullback of the cospan $Pic(R)\xrightarrow{\mathrm{Ind}_R^A}Pic(A)\longleftarrow B(A,A)$ with $B(A,A)$ contractible. Therefore, all functors $\mathrm{Ind}_R^A\circ(f,R_X,\cdots,R_X)$ are equivalent to the constant functor $R_A$, and hence
    \[\mathrm{coTHH}_A(\mathrm{Ind}_R^A(\T f);\mathrm{Ind}_R^A(R[X])) \simeq \mathrm{coTHH}_A(A[X]).\]
    It remains to prove that the left-hand side is equivalent to $\mathrm{Ind}_R^A[\mathrm{coTHH}^R(\T f;R[X])]$, while the right-hand side is equivalent to $A[LX].$ By \cite[Theorem 3.7]{HS21}, there is an equivalence $\mathrm{coTHH}^{\mathbb{S}}(\mathbb{S}[X])\simeq \mathbb{S}[LX]$. Consequently, $A[LX]\simeq A\otimes\mathrm{coTHH}_{\mathbb{S}}(\mathbb{S}[X])$. Now both desired equivalences follow from Proposition \ref{prop:base change}, since $X$ is simply connected.
\end{proof}

Our next aim is to equip $\mathrm{coTHH}^R(\T f;R[X])$ with a filtration induced by the cellular structure of $X$. This provides an approach to compute $\mathrm{coTHH}^R(\T f;R[X])$ that is based on the space's own geometry. Given a space $X$, let $X_n$ denote its $n$-skeleton. Also recall that the space that we consider in the next theorem is simply connected, and hence we can assume that its $0$-skeleton consists of a single point. The following lemma is needed in the proof of Theorem \ref{thm:S1}.

\begin{lem}\label{lem:forget}
	Let $\mathcal{C}$ and $\mathcal{D}$ be symmetric monoidal categories and $F\colon\mathcal{C}\to\mathcal{D}$ be a symmetric monoidal functor. Given a map $A\to B$ of $\mathbb{E}_1$-algebras in $\mathcal{C}$, the following diagram commutes
	    \[
      \xymatrix{
      \mathrm{LMod}_B(\mathcal{C}) \ar[r]^-{U_{\mathcal{C}}} \ar[d]_-{F}      
      &\mathrm{LMod}_A(\mathcal{C}) \ar[d]_-{F}   
      &\\
      \mathrm{LMod}_{F(B)}(\mathcal{D}) \ar[r]^-{U_{\mathcal{D}}}
      &\mathrm{LMod}_{F(A)}(\mathcal{D}),
    }\]
    where $U_{\mathcal{C}}$ and $U_{\mathcal{D}}$ are forgetful functors.
\end{lem}
\begin{proof}
	Since $F$ is symmetric monoidal, there is a commutative diagram
	    \[
      \xymatrix{
      \mathrm{LMod}(\mathcal{C}) \ar[r]^-{F} \ar[d]_-{\theta_{\mathcal{C}}}     
      &\mathrm{LMod}(\mathcal{D}) \ar[d]^-{\theta_{\mathcal{D}}}
      &\\
      \mathrm{Alg}_{\mathbb{E}_1}(\mathcal{C}) \ar[r]^-{F}
      &\mathrm{Alg}_{\mathbb{E}_1}(\mathcal{D}),
    }\]	
where both vertical arrows are forgetful functors, which are Cartesian fibrations by \cite[Corollary 4.2.3.2]{HA}. Consequently, there is a natural transformation 
\[F\circ U_{\mathcal{C}}\Longrightarrow U_{\mathcal{C}}\circ F\]
of functors from $\mathrm{LMod}_B(\mathcal{C})$ to $\mathrm{LMod}_{F(A)}(\mathcal{D})$. To show this natural transformation is an equivalence, it suffices to prove that $F\colon \mathrm{LMod}(\mathcal{C})\to \mathrm{LMod}(\mathcal{D})$ maps $\theta_{\mathcal{C}}$-Cartesian morphisms to $\theta_{\mathcal{D}}$-Cartesian morphisms. Equivalently, by \cite[Corollary 4.2.3.2]{HA}, we need to show that the composition $\mathrm{LMod}(\mathcal{C})\xrightarrow{F} \mathrm{LMod}(\mathcal{D})\to \mathcal{D}$ maps $\theta_{\mathcal{C}}$-Cartesian morphisms to equivalences, where the second functor sends a module to its underlying object. However, this composition is equivalent to $\mathrm{LMod}(\mathcal{C})\to \mathcal{C}\xrightarrow{F}\mathcal{D}$, where the first functor sends a module to its underlying object and maps $\theta_{\mathcal{C}}$-Cartesian morphisms to equivalences, by \cite[Corollary 4.2.3.2]{HA}. 
\end{proof}

\begin{thm}[Theorem \ref{thm:S1}]
	Let $R$ be a connective $\E$-ring and $X$ be a simply connected space. Given a map $X\to B\G$, the $R$-module $\mathrm{coTHH}^R(\T f;R[X])$ admits an exhaustive filtration 
	\[\{\mathrm{coTHH}^R(\T (f|_{X_n});R[X])\}_{n\geq0},\]
    so that $\colim\limits_n \mathrm{coTHH}^R(\T (f|_{X_n});R[X])\simeq \mathrm{coTHH}^R(\T f;R[X])$.
	Furthermore, the $n$-th associated graded of this filtration is given by
	\[gr_n\simeq \bigoplus_{X(n)}\Sigma^{n}R[\Omega X],\ n\geq0,\]
	where $X(n)$ is the number of $n$-cells in $X$.
\end{thm}
\begin{proof}
	For brevity, we denote the restriction of $f$ to $X_n$ by $f_n$. Since $X=\colim\limits_{n}X_n$, there is an equivalence $f\simeq \colim\limits_n f_n$ in $\Ss_{/B\G}$. It follows from Remark \ref{rm:col} that we can write $\T f$ as the colimit
	\begin{equation}\label{eq:colim}
		\T f\simeq \colim\limits_n \colim\limits_{X_n}f \simeq \colim\limits_n\T f_n,
	\end{equation}
	Furthermore, since $X_n=X_{n-1}\sqcup_{\bigvee_{X(n)} S^{n-1}} \bigvee_{X(n)}D^n$, another application of Remark \ref{rm:col} yields the following pushout square
	\[
      \xymatrix{
      \colim\limits_{\bigvee_{X(n)} S^{n-1}}f \ar[r] \ar[d]      &\colim\limits_{\bigvee_{X(n)}D^n}f \ar[d] &\\
      \colim\limits_{X_{n-1}}f \ar[r]
      &\colim\limits_{X_n}f.
    }\]
    Note that $D^n$ is contractible and admits a final object, and the upper horizontal map admits a retraction. Consequently, when restricted to any copy of $S^{n-1}$, the colimit of $f$ is $R[S^{n-1}]$, since it factors through $D^n$. Hence, we obtain that
    \[\colim\limits_{\bigvee_{X(n)} S^{n-1}}f\simeq \bigoplus_{X(n)}R[S^{n-1}]\simeq R\oplus\bigoplus_{X(n)}\Sigma^{n-1} R\simeq \colim\limits_{\bigvee_{X(n)}D^n}f\oplus \bigoplus_{X(n)}\Sigma^{n-1} R,\]
    where the first direct sum is taken in the category $(\M)_{R/}$, and the latter two are taken in $\M$. Therefore, let $X_{-1}$ be the empty set, then the above pushout diagram yields a fiber sequence
    \begin{equation}\label{eq:fib th}
    \bigoplus_{X(n)}\Sigma^{n-1} R \to \T f_{n-1} \to \T f_n   	
    \end{equation}
    for all $n\geq0$. Now, given any pointed maps $f\colon X\to B\G$ and $X\to Y$, by the proof of Theorem \ref{thm:comodule} and the dual of Lemma \ref{lem:forget}, there is a commutative diagram
    \[
      \xymatrix{
      \mathrm{CoMod}_{R_X}(\Ss_{/B\G}) \ar[r]^-{U} \ar[d]_-{\T}      
      &\mathrm{CoMod}_{R_Y}(\Ss_{/B\G}) \ar[d]_-{\T}   
      &\\
      \mathrm{CoMod}_{R[X]}(\M) \ar[r]^-{U}
      &\mathrm{CoMod}_{R[Y]}(\M),
    }\]
    since the Thom spectrum functor is symmetric monoidal by Proposition \ref{prop:sym str}. Hence, unwinding the definitions, the fiber sequence \ref{eq:fib th} is a fiber sequence of $R[X]$-comodules, and the colimit \ref{eq:colim} is a colimit of $R[X]$-comodules. Here the latter step relies on the fact that colimits commute with the tensor product in the category $\Ss_{/B\G}$. 
    
    Furthermore, since the structure maps $S^{n-1}\to X$ factor through $D^n\to X$, we see that the $R[X]$-comodule structure on $\Sigma^{n-1} R$ is trivial, i.e., it lies in the image of the forgetful functor along $R[*]\to R[X]$. Applying the functor $\mathrm{coTHH}^R(-;R[X])$ to the fiber sequence \ref{eq:fib th}, we obtain the fiber sequence
    \begin{equation*}
    	\begin{split}
    		\mathrm{coTHH}^R(\T f_{n-1};R[X]) &\to\mathrm{coTHH}^R(\T f_n;R[X])\\
    		    &\to\Sigma^n\mathrm{coTHH}^R(\bigoplus_{X(n)} R;R[X]).
    	\end{split}
    \end{equation*}
    A similar argument to the proof of Proposition \ref{prop:base change} shows that
    \[\mathrm{coTHH}^R(\bigoplus_{X(n)} R;R[X])\simeq \bigoplus_{X(n)} \mathrm{coTHH}^R(R;R[X]).\]
    Consequently, using Proposition \ref{prop:base change} and Example \ref{exmp:Omega}, the last term in the fiber sequence is equivalent to $\Sigma^n\bigoplus_{X(n)}(R\otimes\mathbb{S}[\Omega X]) \simeq \bigoplus_{X(n)}\Sigma^{n}R[\Omega X]$. Hence, it remains to prove that the natural map
    \[\colim\limits_n\mathrm{coTHH}^R(\T f_n;R[X])\to \mathrm{coTHH}^R(\T f;R[X])\]
    is an equivalence.
    By the same Milnor exact sequence argument as in Proposition \ref{prop:base change}, together with Equation \ref{eq:fib}, we obtain that the map
    \[\mathrm{coTHH}^R(\T f_n;R[X])\to \mathrm{Tot}^i \Omega_R^*(\T f_n,R[X]\otimes_R R[X],R[X])\]
    is $i$-connective. Since colimits preserve connectivity, the map
    \[\colim\limits_n\mathrm{coTHH}^R(\T f_n;R[X])\to \colim\limits_n\mathrm{Tot}^i \Omega_R^*(\T f_n,R[X]\otimes_R R[X],R[X])\]
    is also $i$-connective with right-hand side equivalent to $\mathrm{Tot}^i \Omega_R^*(\T f,R[X]\otimes_R R[X],R[X])$, since $\mathrm{Tot}^i$ is equivalent to a finite limit. Passing to the limit yields the desired equivalence.
\end{proof}

\begin{cor}
	Let $R$ be a connective $\E$-ring, $X$ be a simply connected space and $A$ be a connective $\E$-$R$-algebra. Given a map $f:X\to B\G$, there is a spectral sequence converging  to the $A$-homology of $\mathrm{coTHH}^R(\T f;R[X])$ with $E_1$-page
	\[E_1^{p,q}=\bigoplus_{X(p)}\pi_{p+q}(\Sigma^{p}A[\Omega X])\Longrightarrow \pi_{p+q}(\mathrm{Ind}_R^A(\mathrm{coTHH}^R(\T f;R[X]))).\]
	In particular, if $f$ is $A$-oriented, then we obtain a spectral sequence to compute $A_*(LX)$.
\end{cor}
\begin{proof}
	Recall that for any $\E$-ring $R$, there is a standard $t$-structure on $\M$, where an $R$-module is coconnective if and only if its underlying spectrum is coconnective. Therefore, the standard $t$-structure on $\M$ is compatible with sequential colimits. Now the spectral sequence follows from Theorem \ref{thm:S1} and \cite[Proposition 1.2.2.14]{HA}, and the last claim follows from Theorem \ref{thm:orient}. Alternatively, the colimit is computed in the category $\mathrm{Sp}$, so the same reasoning gives the result.
\end{proof}

\begin{exmp}
	Let $R$ be a connective $\E$-ring equipped with an $S^1$-action in the category $\M$. Equivalently, we can view $R$ as a functor $\mathbb{C}P^{\infty}\to B\G$ and its Thom spectrum is just the $S^1$-orbit $R_{S^1}$. Since $\mathbb{C}P^n$ can be constructed by attaching only even-dimensional cells, using a similar argument as Theorem \ref{thm:S1}, we can obtain a filtration of $\mathrm{coTHH}^R(R_{S^1};R[\mathbb{C}P^{\infty}])$ given by $\{\mathrm{coTHH}^R(R_{\mathbb{C}P^n};R[\mathbb{C}P^{\infty}])\}_{n\geq0}$ with the associated graded
	\[gr_n\simeq \Sigma^{2n}R[\Omega \mathbb{C}P^{\infty}]\simeq \Sigma^{2n}R\oplus\Sigma^{2n+1}R,\ n\geq0.\]
	Indeed, we will see that $\mathrm{coTHH}^R(R_{S^1};R[\mathbb{C}P^{\infty}])\simeq R_{S^1}[\Omega \mathbb{C}P^{\infty}] \simeq R_{S^1} \oplus \Sigma R_{S^1}$ (see Proposition \ref{prop:omega}). Therefore, the filtration above is just the direct sum of the usual filtration on $R_{S^1}$ and the suspension of this filtration.
\end{exmp}

\section{The proof of Theorem \ref{thm:geo}}\label{sec4}

In this section, we give the proof of Theorem \ref{thm:geo}. More precisely, we use Proposition \ref{prop:map} and Theorem \ref{thm:Thf} to reduce the computation of $\mathrm{coTHH}^R(\T f;R[X])$ to the computation of topological coHochschild homology without coefficients.

First, by Theorem \ref{thm:Thf}, there is an equivalence 
\[\T f\otimes_R(R[X])\simeq \T(f,R_X)\simeq R\otimes_{R[G\times G]} R,\]
where the second factor $G$ acts on both copies of $R$ trivially. Hence, to compute $\mathrm{coTHH}^R(\T f;R[X]),$ we also need to describe how the coaction is realized in the aforementioned bar constructions. More generally, given a map $g\colon X\to Y$ between connected spaces, the commutative diagram
\[
    \xymatrix{
    X \ar[r]^-{g} \ar[dr]_-{h}
    & Y \ar[d]^-{f}
    \\
    &Pic(R),
    }
\]
induces a map $g_*\colon\T h\to \T f$ between Thom spectra, or equivalently, a map 
\begin{equation}\label{eq:g_*}
	g_*: R\otimes_{R[\Omega X]}R \to R\otimes_{R[\Omega Y]}R,
\end{equation}
by Theorem \ref{thm:Thf}. The next proposition shows that, as expected, $g_*$ is induced by the map $g$ and the bar constructions of the above two modules.

\begin{prop}\label{cor:map}
	The map in \ref{eq:g_*} is the geometric realization of
	\[\mathrm{Bar}_{\bullet}(id_R,g,id_R):\mathrm{Bar}_{\bullet}(R, R[\Omega X], R) \to \mathrm{Bar}_{\bullet}(R, R[\Omega Y], R).\]
\end{prop}
\begin{proof}
	By Proposition \ref{prop:map} and its proof, for any $R$-module $N$, the map
	\[\mathrm{Map}_{\M^Y}(f,N)\to \mathrm{Map}_{\M^X}(h,N),\]
	is the totalization of the map between cosimplicial objects:
	\[(g^{\times n})^*:\mathrm{Map}_R(R\otimes \Omega Y^{\times n},N) \to \mathrm{Map}_R(R\otimes \Omega X^{\times n},N).\]
	Note that we can also consider $N$ as a right $R[\Omega Y]$-module (or a right $R[\Omega X]$-module, respectively), via the forgetful functor. Hence, we can use the base change adjunction twice: on the left via $R\to R[\Omega Y]$ followed by $R[\Omega Y] \to R$, and on the right via $R\to R[\Omega X]$ followed by $R[\Omega X] \to R$. Combining these with Theorem \ref{thm:Thf}, the Yoneda Lemma yields the desired result.
\end{proof}

Therefore, the coaction of $R[X]$ on $\T f$ is induced by the map
\[\mathrm{Bar}_{\bullet}(id_R,\Delta,id_R):\mathrm{Bar}_{\bullet}(R, R[G], R) \to \mathrm{Bar}_{\bullet}(R, R[G\times G], R).\]
The following two lemmas will be used in the proof of our main theorem.
\begin{lem}\label{lem:con}
	Let $R$ be a connective $\E$-ring and $C$ be a connective coaugmented $R$-coalgebra whose homotopy groups of coaugmentation coideal $\bar{C}$ vanish in degree less than $1$, let $M$ be a connective right $C$-comodule and $N$ be a connective left $C$-comodule, then the fiber of the natural map 
	\[i_n^*:\mathrm{Tot}^{n+1}(\Omega^*_R(M,C,N))\to \mathrm{Tot}^n(\Omega^*_R(M,C,N))\]
    is connective. In particular, $\mathrm{Tot}^n(\Omega^*_R(M,C,N))$ is connective for any $n\geq0$.
\end{lem}
\begin{proof}
	By Equation \ref{eq:fib}, the fiber of $i^*_n$ is equivalent to $\Omega^n(M\otimes_R\overline{C}^{\otimes_R^n}\otimes_R N)$ which is obviously connective.
\end{proof}

Abusing notation, we write $\Tot^n(\mathrm{coTHH}_*^R(M;C))$ for $\Tot^n(\Omega_R^*(M,C\otimes_R C^{op},C))$ from now on. This should not be confused with $\lim_{\Delta_{\leq n}}\mathrm{coTHH^R_*}(M;C)$.
\begin{lem}\label{lem:0-con}
	Let $R$ be a connective $\E$-ring, $C$ be a connective coaugmented $R$-coalgebra whose homotopy groups of coaugmentation coideal $\bar{C}$ vanish in degree less than $1$ and $M$ be a connective $C$-bicomodule, then $R$-modules $\mathrm{coTHH}^R(M;C)$ and $\Tot^n(\mathrm{coTHH}_*^R(M;C))$ are connective.
\end{lem}
\begin{proof}
	By \cite[Theorem 1.1]{Z25}, there is an equivalence 
	\[\mathrm{coTHH}^R(M; C)\simeq M\square_{C\otimes_R C^{op}} C. \]
	Since the $R$-coalgebra $C\otimes_R C^{op}$ is coaugmented with a $1$-connective coaugmentation coideal $\bar{C}\oplus\bar{C}^{op}\oplus(\bar{C}\otimes_R\bar{C}^{op})$, the result follows from Lemma \ref{lem:con} and the Milnor sequence
	\[0 \to {\varprojlim\limits_{n}}^1\pi_{*+1}(\mathrm{Tot}^n) \to \pi_*(\mathrm{coTHH}^R(M;C)) \to \varprojlim\limits_n\pi_*(\mathrm{Tot}^n) \to 0,\]
	here we write $\mathrm{Tot}^n$ for $\mathrm{Tot}^n(\Omega_R^*(M,C\otimes_R C^{op},C)).$
\end{proof}

Before proving the main theorem, recall that given a simplicial object $X_{\bullet}$ in a stable category $\mathcal{C}$,  there is a skeleton filtration $\{sk_n(X_{\bullet})\}$ on the geometric realization of $X_{\bullet}$ with the associated graded
\[gr_n(|X_{\bullet}|)=\Sigma^n(X_n/L_nX),\]
here $L_nX$ is the $n$-th latching object (see \cite[Proposition $\mathrm{VII}$.1.7]{GJ99}) defined as a colimit over a finite poset. Note that the nerve of a finite poset is a finite simplicial set, hence $L_nX$ is a finite colimit.

\begin{thm}[Theorem \ref{thm:geo}]
	Let $R$ be a connective $\E $-ring spectrum and $f\colon X\to B\G$ be a map of spaces with $X$ simply connected, then there is an equivalence
	\[\mathrm{coTHH}^R(\T f;R[X])\simeq \colim\limits_{\Delta^{op}}\mathrm{coTHH}^R(R[G^n])\]
    where $G$ is the loop space of $X$ and $\{\mathrm{coTHH}^R(R[G^n])\}_{n\geq 0}$ is a simplicial $R$-module such that in degree $n$: 
	\begin{enumerate}[(a)]
		\item the face map $d_0$ is induced by the action of $G$ on $R$,
		\item the face maps $d_i$ are induced by the multiplication of $G$ for $0< i < n$,
		\item the face map $d_n$ is induced by the trivial action of $G$ on $R$, or equivalently, by the augmentation map $R[G] \to R$,
		\item the degeneracy maps are induced by the unit map $R\to R[G]$.
	\end{enumerate} 
\end{thm}
\begin{proof}
    By Theorem \ref{thm:comodule} and its proof, the constant functor $R_X$ is an $\E$-coalgebra in $\Ss_{/B\G}$ and $f$ is an $R_X$-comodule. Furthermore, the cyclic cosimplicial object $\mathrm{coTHH}^R_*(\T f;R[X])$ is the composition
    \[\Delta\xrightarrow{\mathrm{coTHH}^{\Ss_{/B\G}}_*(f;R_X)}\Ss_{/B\G}\xrightarrow{\T} \M,\]
    since the Thom spectrum functor is symmetric monoidal by Proposition \ref{prop:sym str}. Using Theorem \ref{thm:Thf} and Proposition \ref{cor:map}, we see that $\mathrm{coTHH}^R_*(\T f;R[X])$ is equivalent to
    \[\Delta\to \mathrm{Fun}(\Delta^{op},\M)\xrightarrow{|-|}\M,\]
    where the first functor is adjoint to the functor
    \begin{equation*}
        \mathrm{Bar}_{\bullet}\mathrm{coTHH}_*(R[G]):\Delta \times \Delta^{op}\longrightarrow \M
    \end{equation*}
    obtained by applying the functor $\mathrm{Bar}^R_{\bullet}(R,-,R)$ to the cosimplicial object $R[\mathrm{coTHH}_*^{\Ss}(G)]$ (note that $R$ is a right $R[G^n]$-module, where the first copy of $G$ acts on $R$ via the given action and the other copies act trivially. Simultaneously, $R$ is also a left $R[G^n]$-module, where every copy of $G$ acts trivially). Informally, $\mathrm{coTHH}^R(\T f;R[X])$ is obtained from the following diagram by first taking the geometric realization of the columns, and then taking the totalization of the resulting cosimplicial object.
    \begin{equation*}
    \begin{tikzcd}[column sep=1.2em]
       \vdots
       \arrow[d, shift left=2.25ex] 
       \arrow[d, shift left=0.75ex] 
       \arrow[d, shift right=0.75ex] 
       \arrow[d, shift right=2.25ex]       
       & \vdots
       \arrow[d, shift left=2.25ex] 
       \arrow[d, shift left=0.75ex] 
       \arrow[d, shift right=0.75ex] 
       \arrow[d, shift right=2.25ex]
       & \vdots
       \arrow[d, shift left=2.25ex] 
       \arrow[d, shift left=0.75ex] 
       \arrow[d, shift right=0.75ex] 
       \arrow[d, shift right=2.25ex] \\
       R\otimes_R R[G\times G]\otimes_R R 
       \arrow[r, shift left=0.5ex] 
       \arrow[r, shift right=0.5ex, swap] 
       \arrow[d, shift left=1.5ex] 
       \arrow[d] 
       \arrow[d, shift right=1.5ex] 
       & R\otimes_R R[(G\times G)^{\times 2}]\otimes_R R 
       \arrow[r, shift left=1.5ex] 
       \arrow[r] 
       \arrow[r, shift right=1.5ex]  
       \arrow[d, shift left=1.5ex] 
       \arrow[d] 
       \arrow[d, shift right=1.5ex] 
       &  R\otimes_R R[(G\times G\times G)^{\times 2}]\otimes_R R
       \arrow[d, shift left=1.5ex] 
       \arrow[d] 
       \arrow[d, shift right=1.5ex] 
       \arrow[r, shift left=2.25ex] 
       \arrow[r, shift left=0.75ex] 
       \arrow[r, shift right=0.75ex] 
       \arrow[r, shift right=2.25ex]
       & \cdots \\
       R\otimes_R R[G]\otimes_R R 
       \arrow[r, shift left=0.5ex] 
       \arrow[r, shift right=0.5ex, swap]
       \arrow[d, shift left=0.5ex] 
       \arrow[d, shift right=0.5ex, swap] 
       & R\otimes_R R[G\times G]\otimes_R R 
       \arrow[r, shift left=1.5ex] 
       \arrow[r] 
       \arrow[r, shift right=1.5ex]
       \arrow[d, shift left=0.5ex] 
       \arrow[d, shift right=0.5ex, swap]     
       & R\otimes_R R[G\times G\times G]\otimes_R R
       \arrow[d, shift left=0.5ex] 
       \arrow[d, shift right=0.5ex, swap]
       \arrow[r, shift left=2.25ex] 
       \arrow[r, shift left=0.75ex] 
       \arrow[r, shift right=0.75ex] 
       \arrow[r, shift right=2.25ex]
       & \cdots \\
       R\otimes_R R 
       \arrow[r, shift left=0.5ex] 
       \arrow[r, shift right=0.5ex, swap] 
       & R\otimes_R R 
       \arrow[r, shift left=1.5ex] 
       \arrow[r] 
       \arrow[r, shift right=1.5ex]     
       & R\otimes_R R 
       \arrow[r, shift left=2.25ex] 
       \arrow[r, shift left=0.75ex] 
       \arrow[r, shift right=0.75ex] 
       \arrow[r, shift right=2.25ex]
       & \cdots \\
   \end{tikzcd}
   \end{equation*}
   
   Now the left-hand side of the comparison map
   \[f:|\mathrm{Tot}(\mathrm{Bar}_{\bullet}\mathrm{coTHH}_*(R[G]))| \to \mathrm{Tot}|\mathrm{Bar}_{\bullet}\mathrm{coTHH}_*(R[G])|\]
   is equivalent to the geometric realization of $\{\mathrm{coTHH}^R(R[G^n])\}_{n\geq0}$, hence it suffices to prove that the comparison map is an equivalence. By the $\infty$-categorical Dold--Kan correspondence (see \cite[Theorem 1.2.4.1]{HA}), there is an equivalence
   \[|\mathrm{Tot}(\mathrm{Bar}_{\bullet}\mathrm{coTHH}_*(R[G]))|\simeq \colim\limits_n[sk_n\mathrm{Tot}(\mathrm{Bar}_{\bullet}\mathrm{coTHH}_*(R[G]))].\]
   Although $N(\Delta^{op}_{\leq n})$ is an infinite simplicial set, there is a cofinal functor $\mathcal{P}([n])^{op}\to \Delta^{op}_{\leq n}$, such that the simplicial set $N(\mathcal{P}([n])^{op})$ is finite. Hence, the comparison map is obtained from the collection of maps
   \[f_n:\Tot\{sk_n[\mathrm{Bar}_{\bullet}\mathrm{coTHH}_*(R[G])]\} \to \Tot|\mathrm{Bar}_{\bullet}\mathrm{coTHH}_*(R[G])|\]
   by passing to the colimit. We claim that the cofiber of $f_n$ is $n$-connective, passing to the colimit gives the desired equivalence. To prove the claim, 
   we consider the associated graded
   \begin{equation*}
   	\begin{split}
   	gr^i_{n+1}
   	   & \simeq \Tot^i sk_{n+1}[\mathrm{Bar}_{\bullet}\mathrm{coTHH}_*(R[G]])/\Tot^i sk_n[\mathrm{Bar}_{\bullet}\mathrm{coTHH}_*(R[G])]\\
   	   & \simeq \Tot^i \{sk_{n+1}[\mathrm{Bar}_{\bullet}\mathrm{coTHH}_*(R[G])]/sk_n[\mathrm{Bar}_{\bullet}\mathrm{coTHH}_*(R[G])]\} \\
   	   & \simeq \Sigma^{n+1} \Tot^i [\mathrm{coTHH}^R_*(R[G^{\times n+1}])/L_{n+1}\mathrm{Bar}_{\bullet}\mathrm{coTHH}_*(R[G])].
   	\end{split}
   \end{equation*}
   Since $\M$ is stable and the $(n+1)$-th latching object is a finite colimit, we have that 
   \[\Tot^i (L_{n+1}\mathrm{Bar}_{\bullet}\mathrm{coTHH}_*(R[G]))\simeq L_{n+1}\Tot^i (\mathrm{Bar}_{\bullet}\mathrm{coTHH}_*(R[G])).\]
   By Lemma \ref{lem:0-con}, $\Tot^i\mathrm{coTHH}^R(R[G^n])$ is connective for $i,n\geq0$, hence, as a colimit of connective spectra, $L_{n+1}\Tot^i(\mathrm{Bar}_{\bullet}\mathrm{coTHH}_*(R[G]))$ is also connective. Therefore, the associated graded $gr^i_{n+1}$ is $(n+1)$-connective and hence so is
   \begin{equation*}
   \begin{split}
    & \Tot^i \{|\mathrm{Bar}_{\bullet}\mathrm{coTHH}_*(R[G])|/sk_n[\mathrm{Bar}_{\bullet}\mathrm{coTHH}_*(R[G])]\} \\
   	& \simeq \Tot^i |\mathrm{Bar}_{\bullet}\mathrm{coTHH}_*(R[G])|/\Tot^i sk_n[\mathrm{Bar}_{\bullet}\mathrm{coTHH}_*(R[G])] \\
   	& \simeq \colim\limits_{m\geq n}\Tot^i sk_m[\mathrm{Bar}_{\bullet}\mathrm{coTHH}_*(R[G])]/\Tot^i sk_n[\mathrm{Bar}_{\bullet}\mathrm{coTHH}_*(R[G])].
   \end{split}   	
   \end{equation*}
   Here we use once again that $\mathrm{Tot}^i$ is a finite limit and therefore commutes with colimits. Passing to the limit and using the Milnor exact sequence
   \begin{equation*}
   	\begin{split}
   		0 & \longrightarrow {\lim\limits_{i}}^1\pi_{*+1}(\Tot^i \{|\mathrm{Bar}_{\bullet}\mathrm{coTHH}_*(R[G])|/sk_n[\mathrm{Bar}_{\bullet}\mathrm{coTHH}_*(R[G])]\})\\
   		& \longrightarrow \pi_*(\Tot |\mathrm{Bar}_{\bullet}\mathrm{coTHH}_*(R[G])|/\Tot\{sk_n[\mathrm{Bar}_{\bullet}\mathrm{coTHH}_*(R[G])]\})\\
   		& \longrightarrow \lim\limits_i\pi_*(\Tot^i \{|\mathrm{Bar}_{\bullet}\mathrm{coTHH}_*(R[G])|/sk_n[\mathrm{Bar}_{\bullet}\mathrm{coTHH}_*(R[G])]\}) \longrightarrow 0,
   	\end{split}
   \end{equation*}
   we deduce that the cofiber of $f_n$ is $n$-connective, which proves the claim.
\end{proof}

As a corollary, we show that the computation becomes simpler when $X$ is sufficiently highly connected.

\begin{cor}
	Let $R$ be a connective $\E $-ring spectrum and $f\colon X\to B\G$ be a map of spaces with $X$ $2$-connected, then $\mathrm{coTHH}^R(\T f;R[X])$ is the Thom spectrum of $\mathrm{coTHH}^{\Ss_{/B\G}}(f;R_X)$. In other words,
	\[\mathrm{coTHH}^R(\T f;R[X])\simeq \T (\mathrm{coTHH}^{\Ss_{/B\G}}(f;R_X)).\]
\end{cor}
\begin{proof}
	Since $X$ is $2$-connected, $\Omega X$ is simply connected. Therefore, by \cite[Theorem 3.7]{HS21}, we have that 
	\[\mathrm{coTHH}^R(R[G^n])\simeq R[LG^n]\simeq R[LG]\otimes_R\cdots\otimes_R R[LG],\]
	and hence, by Theorem \ref{thm:geo}, we obtain that
	\begin{equation}\label{eq:LG}
	\mathrm{coTHH}^R(\T f;R[X])\simeq R\otimes_{R[LG]}R.		
	\end{equation}
	Note that $LX$ is the pullback of the cospan $X\xrightarrow{\Delta} X\times X \xleftarrow{\Delta} X$, hence $\Omega LX\simeq L\Omega X$. Consequently,  $\mathrm{coTHH}^R(\T f;R[X])$ itself is a Thom spectrum of the map $\mathrm{coTHH}^{\Ss_{/B\G}}(f;R_X)\in \Ss_{/B\G}$ by Example \ref{exmp:over BG}, and the equivalence follows from Equivalence \ref{eq:LG} and Theorem \ref{thm:Thf}.
\end{proof}

\begin{exmp}\label{exm:T and Omega}
	If we further assume that $X$ is an $\mathbb{E}_1$-space, then the map $X\times \Omega X\to LX$ is an equivalence, and $\mathrm{coTHH}^{\Ss_{/B\G}}(f;R_X)\in \Ss_{/B\G}$ factors as
	\[\mathrm{coTHH}^{\Ss_{/B\G}}(f;R_X): LX\simeq X\times \Omega X \xrightarrow{pr_1} X\to B\G,\]
	and hence $\mathrm{coTHH}^R(\T f;R[X])\simeq \T f[\Omega X]$.	
\end{exmp}

In the above theorem, we view the Thom spectrum as a relative tensor product. This allows us to use its skeleton filtration to transform the computation of $\mathrm{coTHH}$ with a Thom spectrum coefficient into a problem involving two steps: first, computing $\mathrm{coTHH}$ of simplicial coalgebras, and second, taking the geometric realization of the resulting simplicial object. A limitation of this approach is that it lowers the connectivity of spaces by $1$ (taking the loop spaces). Consequently, we may not use the coB\"okstedt spectral sequence to compute $A$-homology of $\mathrm{coTHH}^R(R[G^n])$. A detailed treatment of this aspect is provided in \cite[Section 4]{BGT+18}. Nonetheless, the following corollary allows us to bypass this problem.

\begin{cor}
	Under the assumption of Theorem \ref{thm:geo}, let $A$ be a connective $\E$-$R$-algebra, then there is an equivalence
	 \[\mathrm{Ind}^A_R(\mathrm{coTHH}^R(\T f;R[X]))\simeq |\mathrm{coTHH}^A(A[G^n])|,\]
	  and face maps and degeneracy maps of $\{\mathrm{coTHH}^A(A[G^n])\}_{n\geq0}$ are similar to those of Theorem \ref{thm:geo}.
\end{cor}
\begin{proof}
	Since $\mathrm{Ind}^A_R(\T f)$ is the Thom spectrum of the functor $X\xrightarrow{f} Pic(R)\xrightarrow{-\otimes_R A} Pic(A)$, the corollary follows immediately from Proposition \ref{prop:base change}, Remark \ref{rem:ind} and Theorem \ref{thm:geo}.
\end{proof}

Given a map between simply connected spaces $X\to Y$, we can view $\T f$ as an $R[Y]$-comodule via the forgetful functor. Furthermore, let $G$ and $G^{\prime}$ be loop spaces of $X$ and $Y$ respectively, then $R[G]$ is an $R[G^{\prime}]$-comodule. A proof similar to that of Theorem \ref{thm:geo} gives the following theorem.

\begin{thm}
	Let $R$ be a connective $\E$-ring, $X\to Y$ be a map between simply connected spaces and $f:X\to Pic(R)$ be a based map, then there is an equivalence
	\[\mathrm{coTHH}^R(\T f;R[Y])\simeq |\mathrm{coTHH}^R(R[G];R[(G^{\prime})^n])|,\]
	and face maps and degeneracy maps of $\{\mathrm{coTHH}^R(R[G];R[(G^{\prime})^n])\}_{n\geq0}$ are similar to those of Theorem \ref{thm:geo}.
\end{thm}
\begin{proof}
	The result follows by applying the proof of Theorem \ref{thm:geo} to the equivalence $\T(f,R_Y,\cdots,R_Y)\simeq R \otimes_{R[\Omega X\times \cdots\times\Omega Y]}R$ of Theorem \ref{thm:Thf}.
\end{proof}

Finally, we consider the case where the simply connected space $X$ admits an $\mathbb{E}_1$-structure. In this case, $\mathrm{coTHH}^R(\T f;R[X])$ admits some additional structure. For example, combining with Example \ref{exmp:Omega}, we obtain a natural equivalence
\begin{equation*}
	R[\Omega X]\simeq \mathrm{coTHH}^R(R;R[X]).
\end{equation*}
Furthermore, the multiplication on $X$ induces a map
\[\mathrm{coTHH}^R(R[X])\otimes_R \mathrm{coTHH}^R(R;R[X])\simeq R[LX\times \Omega X]\to R[L X]\simeq \mathrm{coTHH}^R(R[X]).\]
Consequently, $\mathrm{coTHH}^R(R[X])$ is a right $\mathrm{coTHH}^R(R;R[X])$-module. Note that the multiplication on $\Omega X$ here is induced by that on $X$ rather than the composition of loops. Nevertheless, they are identical by an Eckmann-Hilton argument. Similarly, Example \ref{exm:T and Omega} shows that $\mathrm{coTHH}^R(\T f;R[X])$ is also a right $R[\Omega X]$-module, when $X$ is a $2$-connected $\mathbb{E}_1$-space. We end the paper by proving the following more general result. That is, the last formula in Example \ref{exm:T and Omega} is true when the $\mathbb{E}_1$-space $X$ is only simply connected.

\begin{prop}\label{prop:omega}
	Let $R$ be a connective $\E$-ring and $X$ be a simply connected $\mathbb{E}_1$-space. Given a map of spaces $f\colon X\to B\G$, there is an equivalence
	\[\mathrm{coTHH}^R(\T f;R[X])\simeq \T f[\Omega X].\]
\end{prop}
\begin{proof}
	By \cite[Proposition 4.6]{Z25}, there is an equivalence 
	\[\mathrm{coTHH}^R(\T f;R[X])\simeq \T f\square_{R[X]}\mathrm{coTHH}^R(R[X])\simeq \T f\square_{R[X]}R[LX].\]
	By Example \ref{exm:T and Omega}, the $R[X]$-comodule structure on $R[LX]$ is induced by the map $LX\simeq X\times \Omega X\xrightarrow{pr_1} X$. Hence, as an $R[X]$-comodule, $R[LX]$ is equivalent to the tensor product of $(R[X],R[X])$ and $(R[\Omega X],R)$ in the category $\mathrm{CoMod}(\M)$ (see \cite[Proposition 3.2.4.3]{HA}). Now, by the splitting of the cobar construction $\Omega^{\bullet}_R(\T f,R[X],R[X])$, we obtain that 
	\[\mathrm{coTHH}^R(\T f;R[X])\simeq \T f\square_{R[X]}R[LX]\simeq \T f\square_{R[X]}(R[X]\otimes R[\Omega X])\simeq \T f[\Omega X].\]
\end{proof}

\begin{cor}
	Let $R$ be a connective $\E$-ring and $X$ be a simply connected $\mathbb{E}_1$-space. Given a map of spaces $f\colon X\to B\G$, the $R$-module $\mathrm{coTHH}^R(\T f;R[X])$ is a right $R[\Omega X]$-module in the category $\M$.
\end{cor}

\bibliographystyle{plain}
\bibliography{ref}

@book{HTT,
  title={Higher topos theory},
  author={Lurie, Jacob},
  year={2009},
  publisher={Princeton University Press}
}

@book{HA,
 author = {Lurie, Jacob},
 title = {Higher Algebra},
 year = {2017},
 note = {Preprint available at \url{www.math.ias.edu/~lurie/}},
}

@Article{BP23,
 Author = {Bayindir, Haldun {\"O}zg{\"u}r and P{\'e}roux, Maximilien},
 Title = {Spanier-{Whitehead} duality for topological {coHochschild} homology},
 FJournal = {Journal of the London Mathematical Society. Second Series},
 Journal = {J. Lond. Math. Soc., II. Ser.},
 ISSN = {0024-6107},
 Volume = {107},
 Number = {5},
 Pages = {1780--1822},
 Year = {2023},
 Language = {English},
 DOI = {10.1112/jlms.12725},
 zbMATH = {7731088},
 Zbl = {1550.16031}
}

@article{NS18,
 author = {Nikolaus, Thomas and Scholze, Peter},
 title = {On topological cyclic homology},
 fjournal = {Acta Mathematica},
 journal = {Acta Math.},
 issn = {0001-5962},
 volume = {221},
 number = {2},
 pages = {203--409},
 year = {2018},
 language = {English},
 doi = {10.4310/ACTA.2018.v221.n2.a1},
 zbMATH = {7009201},
 Zbl = {1457.19007}
}

@article{Pe22,
 author = {P{\'e}roux, Maximilien},
 title = {The coalgebraic enrichment of algebras in higher categories},
 fjournal = {Journal of Pure and Applied Algebra},
 journal = {J. Pure Appl. Algebra},
 issn = {0022-4049},
 volume = {226},
 number = {3},
 pages = {11},
 note = {Id/No 106849},
 year = {2022},
 language = {English},
 doi = {10.1016/j.jpaa.2021.106849},
 zbMATH = {7396430},
 Zbl = {1484.16040}
}

@article{Doi81,
 author = {Doi, Yukio},
 title = {Homological coalgebra},
 fjournal = {Journal of the Mathematical Society of Japan},
 journal = {J. Math. Soc. Japan},
 issn = {0025-5645},
 volume = {33},
 pages = {31--50},
 year = {1981},
 language = {English},
 doi = {10.2969/jmsj/03310031},
 zbMATH = {3719314},
 Zbl = {0459.16007}
}

@article{HPS09,
 author = {Hess, Kathryn and Parent, Paul-Eug{\`e}ne and Scott, Jonathan},
 title = {{CoHochschild} homology of chain coalgebras.},
 fjournal = {Journal of Pure and Applied Algebra},
 journal = {J. Pure Appl. Algebra},
 issn = {0022-4049},
 volume = {213},
 number = {4},
 pages = {536--556},
 year = {2009},
 language = {English},
 doi = {10.1016/j.jpaa.2008.08.001},
 zbMATH = {5527480},
 Zbl = {1238.16005}
}

@article{BGS22,
 author = {Bohmann, Anna Marie and Gerhardt, Teena and Shipley, Brooke},
 title = {Topological {coHochschild} homology and the homology of free loop spaces},
 fjournal = {Mathematische Zeitschrift},
 journal = {Math. Z.},
 issn = {0025-5874},
 volume = {301},
 number = {1},
 pages = {411--454},
 year = {2022},
 language = {English},
 doi = {10.1007/s00209-021-02879-4},
 zbMATH = {7507822},
 Zbl = {1493.55010}
}

@article{Mal17,
 author = {Malkiewich, Cary},
 title = {Cyclotomic structure in the topological {Hochschild} homology of {{\(DX\)}}},
 fjournal = {Algebraic \& Geometric Topology},
 journal = {Algebr. Geom. Topol.},
 issn = {1472-2747},
 volume = {17},
 number = {4},
 pages = {2307--2356},
 year = {2017},
 language = {English},
 doi = {10.2140/agt.2017.17.2307},
 zbMATH = {6762692},
 Zbl = {1390.19005}
}

@article{CCRY25,
 author = {Carmeli, Shachar and Cnossen, Bastiaan and Ramzi, Maxime and Yanovski, Lior},
 title = {Characters and transfer maps via categorified traces},
 fjournal = {Forum of Mathematics, Sigma},
 journal = {Forum Math. Sigma},
 issn = {2050-5094},
 volume = {13},
 pages = {84},
 note = {Id/No e93},
 year = {2025},
 language = {English},
 doi = {10.1017/fms.2025.23},
 zbMATH = {8055054}
}

@article{ABG+,
	author = {Ando, Matthew and Blumberg, Andrew J. and Gepner, David and Hopkins, Michael J. and Rezk, Charles},
	title = {An {{\(\infty\)}}-categorical approach to {{\(R\)}}-line bundles, {{\(R\)}}-module {Thom} spectra, and twisted {{\(R\)}}-homology},
	fjournal = {Journal of Topology},
	journal = {J. Topol.},
	issn = {1753-8416},
	volume = {7},
	number = {3},
	pages = {869--893},
	year = {2014},
	language = {English},
	doi = {10.1112/jtopol/jtt035},
	zbMATH = {6349724},
	Zbl = {1312.55011}
}

@article{HY17,
	author = {Horev, Asaf and Yanovski, Lior},
	title = {On conjugates and adjoint descent},
	fjournal = {Topology and its Applications},
	journal = {Topology Appl.},
	issn = {0166-8641},
	volume = {232},
	pages = {140--154},
	year = {2017},
	language = {English},
	doi = {10.1016/j.topol.2017.10.007},
	zbMATH = {6804700},
	Zbl = {1387.55024}
}

@article{Bea23,
 author = {Beardsley, Jonathan},
 title = {On bialgebras, comodules, descent data and {Thom} spectra in {{\(\infty\)}}-categories},
 fjournal = {Homology, Homotopy and Applications},
 journal = {Homology Homotopy Appl.},
 issn = {1532-0073},
 volume = {25},
 number = {2},
 pages = {219--242},
 year = {2023},
 language = {English},
 doi = {10.4310/HHA.2023.v25.n2.a10},
 zbMATH = {7771515},
 Zbl = {1535.16041}
}

@misc{Z25,
 author = {Zha, Jiaxi},
 title = {On topological {coHochschild} homology and cotensor products},
 year = {2025},
 howpublished = {Preprint, {arXiv}:2505.14474 [math.{AT}] (2025)},
 url = {https://arxiv.org/abs/2505.14474},
 arXiv = {arXiv:2505.14474}
}

@article{HS21,
 author = {Hess, Kathryn and Shipley, Brooke},
 title = {Invariance properties of {coHochschild} homology},
 fjournal = {Journal of Pure and Applied Algebra},
 journal = {J. Pure Appl. Algebra},
 issn = {0022-4049},
 volume = {225},
 number = {2},
 pages = {26},
 note = {Id/No 106505},
 year = {2021},
 language = {English},
 doi = {10.1016/j.jpaa.2020.106505},
 zbMATH = {7241711},
 Zbl = {1454.13018}
}

@misc{KMN23,
 author = {Krause, Achim and McCandless, Jonas and Nikolaus, Thomas},
 title = {Polygonic spectra and {TR} with coefficients},
 year = {2023},
 howpublished = {Preprint, {arXiv}:2302.07686 [math.{AT}] (2023)},
 url = {https://arxiv.org/abs/2302.07686},
 arXiv = {arXiv:2302.07686}
}

@book{GJ99,
 author = {Goerss, Paul G. and Jardine, John F.},
 title = {Simplicial homotopy theory},
 fseries = {Progress in Mathematics},
 series = {Prog. Math.},
 issn = {0743-1643},
 volume = {174},
 isbn = {3-7643-6064-X},
 year = {1999},
 publisher = {Basel: Birkh{\"a}user},
 language = {English},
 zbMATH = {1329189},
 Zbl = {0949.55001}
}

@article{BGT+18,
 author = {Bohmann, Anna Marie and Gerhardt, Teena and H{\o}genhaven, Amalie and Shipley, Brooke and Ziegenhagen, Stephanie},
 title = {Computational tools for topological {coHochschild} homology},
 fjournal = {Topology and its Applications},
 journal = {Topology Appl.},
 issn = {0166-8641},
 volume = {235},
 pages = {185--213},
 year = {2018},
 language = {English},
 doi = {10.1016/j.topol.2017.12.008},
 zbMATH = {6836781},
 Zbl = {1390.19004}
}

@article{AB19,
 author = {Antol{\'{\i}}n-Camarena, Omar and Barthel, Tobias},
 title = {A simple universal property of {Thom} ring spectra},
 fjournal = {Journal of Topology},
 journal = {J. Topol.},
 issn = {1753-8416},
 volume = {12},
 number = {1},
 pages = {56--78},
 year = {2019},
 language = {English},
 doi = {10.1112/topo.12084},
 zbMATH = {7055379},
 Zbl = {1417.55007}
}

@article{KY97,
 author = {Kuribayashi, Katsuhiko and Yamaguchi, Toshihiro},
 title = {The cohomology algebra of certain free loop spaces},
 fjournal = {Fundamenta Mathematicae},
 journal = {Fundam. Math.},
 issn = {0016-2736},
 volume = {154},
 number = {1},
 pages = {57--73},
 year = {1997},
 language = {English},
 url = {https://eudml.org/doc/212227},
 zbMATH = {1083229},
 Zbl = {0883.55006}
}

@article{Blu10-1,
 author = {Blumberg, Andrew J.},
 title = {Topological {Hochschild} homology of {Thom} spectra which are {{\( E_{\infty} \)}}-ring spectra},
 fjournal = {Journal of Topology},
 journal = {J. Topol.},
 issn = {1753-8416},
 volume = {3},
 number = {3},
 pages = {535--560},
 year = {2010},
 language = {English},
 doi = {10.1112/jtopol/jtq017},
 zbMATH = {5793331},
 Zbl = {1214.19003}
}

@article{Blu10-2,
 author = {Blumberg, Andrew J. and Cohen, Ralph L. and Schlichtkrull, Christian},
 title = {Topological {Hochschild} homology of {Thom} spectra and the free loop space},
 fjournal = {Geometry \& Topology},
 journal = {Geom. Topol.},
 issn = {1465-3060},
 volume = {14},
 number = {2},
 pages = {1165--1242},
 year = {2010},
 language = {English},
 doi = {10.2140/gt.2010.14.1165},
 zbMATH = {5730921},
 Zbl = {1219.19006}
}

@article{Ber24,
 author = {Berman, John D.},
 title = {On lax limits in {{\(\infty\)}}-categories},
 fjournal = {Proceedings of the American Mathematical Society},
 journal = {Proc. Am. Math. Soc.},
 issn = {0002-9939},
 volume = {152},
 number = {12},
 pages = {5055--5066},
 year = {2024},
 language = {English},
 doi = {10.1090/proc/16968},
 zbMATH = {7962886},
 Zbl = {1555.18035}
}

@misc{KP25-2,
 author = {Keenan, Liam and P{\'e}roux, Maximilien},
 title = {On products of skeleta},
 year = {2025},
 howpublished = {Preprint, {arXiv}:2510.18961 [math.{AT}] (2025)},
 url = {https://arxiv.org/abs/2510.18961},
 arXiv = {arXiv:2510.18961}
}
	
\end{document}